\documentclass[11pt,a4paper,reqno]{amsart}

\usepackage[left=2cm,right=2cm]{geometry}
\usepackage{amsmath,amsfonts,amssymb,amsthm}
\usepackage{mathtools}
\usepackage{epsfig}
\usepackage{graphicx}
\usepackage{url}
\usepackage[usenames,dvipsnames]{xcolor}
\usepackage[colorlinks=true,linkcolor=Blue,citecolor=Green,unicode]{hyperref}
\usepackage{nicefrac}
\usepackage{aliascnt}       
\usepackage[alwaysadjust]{paralist}
\usepackage[T1]{fontenc}
\usepackage[utf8]{inputenc}
\usepackage{tikz,xifthen}
\usepackage{accents}
\usepackage{enumitem}
\usepackage{framed}
\setlist[enumerate]{leftmargin=*}

\usepackage[linesnumbered,boxed]{algorithm2e}
\usepackage{caption}
\usepackage[labelfont=rm]{subcaption}

\usepackage{blkarray}

\usetikzlibrary{matrix}

\usepackage[textwidth=2.5cm,textsize=small]{todonotes}

\newtheorem{theorem}{Theorem}[section]
\newaliascnt{lemma}{theorem}
\newaliascnt{corollary}{theorem}
\newaliascnt{definition}{theorem}
\newaliascnt{remark}{theorem}
\newaliascnt{proposition}{theorem}
\newaliascnt{conjecture}{theorem}
\newaliascnt{example}{theorem}
\newaliascnt{question}{theorem}
\newaliascnt{claim}{theorem}

\newtheorem{lemma}[lemma]{Lemma}
\aliascntresetthe{lemma}

\newtheorem*{lemma*}{Lemma}

\newtheorem{corollary}[corollary]{Corollary}
\aliascntresetthe{corollary}

\newtheorem*{corollary*}{Corollary}

\newtheorem{definition}[definition]{Definition}
\aliascntresetthe{definition}

\newtheorem*{definition*}{Definition}

\newtheorem{remark}[remark]{Remark}
\aliascntresetthe{remark}

\newtheorem*{remark*}{Remark}

\newtheorem{proposition}[proposition]{Proposition}
\aliascntresetthe{proposition}

\newtheorem*{proposition*}{Proposition}

\aliascntresetthe{conjecture}

\newtheorem*{conjecture*}{Conjecture}

\newtheorem{example}[example]{Example}
\aliascntresetthe{example}

\newtheorem*{example*}{Example}

\newtheorem*{problem*}{Problem}

\aliascntresetthe{question}

\newtheorem*{question*}{Question}

\aliascntresetthe{claim}

\newtheorem*{claim*}{Claim}

\DeclareMathOperator{\inter}{int}

\DeclareMathOperator{\conv}{conv}

\DeclareMathOperator{\rc}{rc}

\DeclareMathOperator{\obs}{Obs}

\def\D{\mathcal{D}}

\def\cM{\mathcal{M}}
\def\cO{\mathcal{O}}

\def\cS{\mathcal{S}}

\def\R{\mathbb{R}}

\def\Z{\mathbb{Z}}

\def\N{\mathbb{N}}
\def\Q{\mathbb{Q}}

\DeclareMathOperator{\vol}{vol}

\DeclareMathOperator{\aff}{aff}
\DeclareMathOperator{\lin}{lin}

\DeclareMathOperator{\rec}{rec}
\DeclarePairedDelimiter{\card}{\lvert}{\rvert}

\newcommand{\setcond}[2]{\left\{ #1 \,:\, #2 \right\}}
\mathchardef\mhyphen="2D

\def\true{\mathrm{true}}
\def\false{\mathrm{false}}
\DeclareMathOperator{\BCPI}{BCPI}
\DeclareMathOperator{\BCLI}{BCLI}
\DeclareMathOperator{\ca}{ca}

\newcommand{\floor}[1]{\left\lfloor #1\right\rfloor}
\newcommand{\ceil}[1]{\left\lceil #1 \right\rceil}

\renewenvironment{framed}[1][\hsize]
   {\MakeFramed{\hsize#1\advance\hsize-\width \FrameRestore}}%
   {\endMakeFramed}

\DeclareFontFamily{U}{mathx}{\hyphenchar\font45}
\DeclareFontShape{U}{mathx}{m}{n}{
      <5> <6> <7> <8> <9> <10>
      <10.95> <12> <14.4> <17.28> <20.74> <24.88>
      mathx10
      }{}
\DeclareSymbolFont{mathx}{U}{mathx}{m}{n}
\DeclareFontSubstitution{U}{mathx}{m}{n}
\DeclareMathAccent{\widebar}{0}{mathx}{"73}

\makeatletter
\def\input@path{{./Images/}{./}}
\makeatother

\begin{document}

\title{Complexity of linear relaxations in integer programming}

\author{Gennadiy Averkov}
\address{BTU Cottbus-Senftenberg\\
  Platz der Deutschen Einheit 1\\
  03046 Cottbus\\
  Germany}
\email{averkov@b-tu.de}

\author{Matthias Schymura}
\address{BTU Cottbus-Senftenberg\\
  Platz der Deutschen Einheit 1\\
  03046 Cottbus\\
  Germany}
\email{schymura@b-tu.de}

\thanks{The second author was partially supported by the Swiss National Science Foundation (SNSF) within the project \emph{Lattice Algorithms and Integer Programming (Nr.~185030)}.}




\begin{abstract}
For a set $X$ of integer points in a polyhedron, the smallest number of facets of any polyhedron whose set of integer points coincides with~$X$ is called the relaxation complexity~$\rc(X)$.
This parameter was introduced by Kaibel \& Weltge (2015) and captures the complexity of linear descriptions of~$X$ without using auxiliary variables.

Using tools from combinatorics, geometry of numbers, and quantifier elimination, we make progress on several open questions regarding $\rc(X)$ and its variant $\rc_\Q(X)$, restricting the descriptions of~$X$ to rational polyhedra.

As our main results we show that $\rc(X) = \rc_\Q(X)$ when: (a) $X$ is at most four-dimensional, (b) $X$ represents every residue class in $(\Z/2\Z)^d$, (c) the convex hull of~$X$ contains an interior integer point, or (d) the lattice-width of~$X$ is above a certain threshold.
Additionally, $\rc(X)$ can be algorithmically computed when~$X$ is at most three-dimensional, or $X$ satisfies one of the conditions (b), (c), or (d) above.
Moreover, we obtain an improved lower bound on $\rc(X)$ in terms of the dimension of~$X$.
\end{abstract}

\maketitle

\section{Introduction}

Encoding discrete optimization instances as integer vectors satisfying a system of linear constraints is a fundamental principle of combinatorial optimization, successfully applied in a multitude of cases in the last six decades.  This approach establishes a connection to integer programming and, by this, allows to use linear-programming based solution techniques. In the course of development,  classical combinatorial optimization problems have been endowed with standard formulations as integer programs. For example, the formulation based on the subtour elimination constraints is the standard formulation of the traveling salesman problem. Nevertheless, we would like to draw attention to the fact that if a certain discrete set of feasible solutions is represented as a set $X$ of integer vectors, then a priori there are many different ways to describe $X$ via a system of linear inequalities in integer variables. Thus, it makes sense to investigate the family of \emph{all} possible representations of $X$ using integer-programming constraints in order to detect the most ``interesting'' ones. So far, there have been many results about the tight descriptions based on the facet-defining inequalities for the convex hull of $X$. The interest in tight descriptions is easily motivated by the fact that knowing the inequalities that describe the convex hull of $X$ allows to optimize a linear objective over $X$ \emph{exactly} using linear programming.
However, since tight descriptions can have a very large size, we believe that one should not only focus on the tightness but also investigate possibilities of finding formulations of small size.
While the size of the description is determined by the number of constraints and the size of the coefficients, here we only address problems related to the number of constraints (regarding the size of the coefficients, we refer to the recent work \cite{hojny2018strong} of Hojny). 

We work at the level of general integer programming, which means that we consider an arbitrary finite set $X \subseteq \Z^d$ of integer points in a polyhedron and investigate possibilities to describe this set within the integer lattice $\Z^d$ by a possibly small number of linear inequalities in integer variables. Our research is motivated by a recent contribution of Kaibel \& Weltge \cite{kaibelweltge2015lowerbounds} (see also Weltge's PhD thesis \cite{weltge2015diss}), who posed a number of fundamental problems in this context. We also point out that, in a similar spirit, descriptions of sets of $0/1$ points within the discrete hypercube $\{0,1\}^d$ were considered by Jeroslow \cite{jeroslow1975ondefining} in the 1970s and are an important subject in the theory of social choice~\cite{hammeribarakipeled1981threshold,taylorzwicker1999simplegames}.

We now introduce some notation and present our results.

\begin{definition}
	For a system of linear constraints $Ax \le b, \ U x = v$ with real coefficients, consider
	\begin{equation} \label{X:descr}
		X = \setcond{x \in \Z^d}{A x \le b, \ U x = v}.
	\end{equation} 
	We call the system $A x \le b, \ U x=v$, as well as the polyhedron that it defines, a \emph{relaxation} of $X$ within the integer lattice $\Z^d$. 
	
	When \eqref{X:descr} holds and $A x \le b$ is a system of $k$ inequalities, we say that \emph{$X$ is described by $k$ linear inequalities within $\Z^d$}. If the coefficients of $A,U,b,v$ are rational numbers, we say that \emph{$X$ is described by~$k$ rational linear inequalities within $\Z^d$}. See Figure~\ref{fig:square-descriptions} for an illustration.
	
	The minimal $k$ such that $X$ can be described by~$k$ linear inequalities (resp. rational linear inequalities) within $\Z^d$ is called the \emph{relaxation complexity} of $X$ (resp.~\emph{rational relaxation complexity} of $X$) within~$\Z^d$ and denoted by $\rc(X)$ (resp.~$\rc_\Q(X)$).
\end{definition}

\begin{figure}[ht]
	\hfill\includegraphics[scale=.8,page=1]{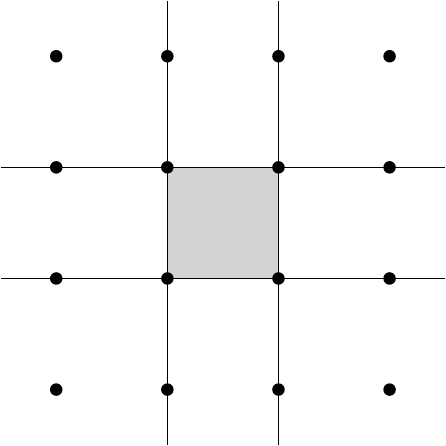}
	\hfill\includegraphics[scale=.8,page=2]{square-descriptions.pdf}
	\hfill\,
	\caption{Descriptions of $\{0,1\}^2$ by four and three (rational) linear inequalities within $\Z^2$.}
	\label{fig:square-descriptions}
\end{figure}

We call a set $X \subseteq \Z^d$ satisfying $\conv(X) \cap \Z^d = X$ \emph{lattice-convex}.
It is easy to see that a finite subset $X$ of $\Z^d$ has a relaxation within $\Z^d$ if and only if $X$ is lattice-convex.
Since in combinatorial optimization one is usually interested in finite subsets $X \subseteq \Z^d$, we deal with finite lattice-convex sets in the sequel.

The two values $\rc(X)$ and $\rc_\Q(X)$ were recently introduced by Kaibel \& Weltge~\cite{kaibelweltge2015lowerbounds}. In terms of polyhedra, $\rc(X)$ and $\rc_\Q(X)$ is the minimum number of facets of a polyhedron (resp.~rational polyhedron) $P$ satisfying $X = P \cap \Z^d$.
Kaibel \& Weltge posed the following three general questions for arbitrary finite lattice-convex sets:

\begin{enumerate}[label=(Q\arabic*)]
	\item \label{q1} Does the inequality $\rc(X) \geq \dim(X) + 1$ hold?
	\item \label{q2} Do $\rc(X)$ and $\rc_\Q(X)$ coincide?
	\item \label{q3} Are $\rc(X)$ and $\rc_\Q(X)$ algorithmically computable?
\end{enumerate}

We note that a positive answer to \ref{q2} implies a positive answer to \ref{q1}, because every rational relaxation of~$X$ is necessarily bounded within the affine hull of~$X$, and thus $\rc_\Q(X) \ge \dim(X)+1$.
On the other hand, a positive answer to \ref{q2} simplifies but does not automatically resolve \ref{q3}, because~\ref{q3} is open for both $\rc(X)$ and $\rc_\Q(X)$.
Regarding \ref{q3} we point out that this question is open to the extent that we do not even know $\rc(X)$ for very concrete and simple looking examples, like the set $\Delta_d = \{0,e_1,\ldots,e_d\}$ of lattice points of the standard simplex.

Our first contribution concerns \ref{q1} and provides an improved lower bound on the relaxation complexity in terms of the dimension:

\begin{theorem}
	\label{thm:rc-vs-dim-intro}
	Let $X \subseteq \Z^d$ be a finite lattice-convex set of dimension $\dim(X) \geq 4$. Then,
	\[
	\rc(X) > \log_2 (\dim(X)) - \log_2 \log_2 (\dim(X)).
	\]
\end{theorem}

We prove this bound in Section~\ref{sect:rc-vs-dim}.
In Section~\ref{sect:small-dimensions}, we show that for $d\leq 4$, Question~\ref{q1} can be answered affirmatively.

Weltge~\cite[Sect.~7.5]{weltge2015diss} showed that in dimension two, $\rc(X)$ and $\rc_\Q(X)$ coincide and are computable.
Already, passing from dimension two to dimension three, both Questions~\ref{q2} and~\ref{q3} get considerably harder.
As our main contributions, we give an affirmative answer to both questions in various settings:
First, we successfully treat small dimensions in full generality.

\begin{theorem}
	\label{thm:main-low-dims} 
	Let $X$ be a finite lattice-convex subset of $\Z^d$.
	\begin{enumerate}[label=(\alph*)]
		\item \label{d-leq-4} If $d \leq 4$, then $\rc(X) = \rc_\Q(X)$.
		\item \label{d=3} If $d\leq 3$, then $\rc(X)$ and $\rc_\Q(X)$ can be computed algorithmically.
	\end{enumerate}
\end{theorem}

Second, we identify large families of sets~$X$ in arbitrary dimensions that are defined by natural conditions and for which Questions~\ref{q2} and~\ref{q3} can be answered positively.

\begin{theorem}
	\label{thm:main-conds} 
	Let $X$ be a finite lattice-convex subset of $\Z^d$ satisfying one of the following conditions:
	\begin{enumerate}[label=(\alph*)]
		\item \label{par:complete} The points in $X$ represent every residue class in $(\Z/2\Z)^d$.
		\item \label{int:cond} The polytope $\conv(X)$ contains interior lattice points.
		\item \label{large:width} The lattice-width of $X$ is bigger than the finiteness threshold width $w^\infty(d)$, introduced in~\cite{blancohaasehofmannsantos2016finiteness}.
	\end{enumerate}
	Then, $\rc(X)$ and $\rc_\Q(X)$ coincide and can be computed algorithmically.
\end{theorem}

The value of $w^\infty(d)$ is known in small dimensions.
We have $w^\infty(1) = w^\infty(2) = 0$, $w^\infty(3) = 1$, and $w^\infty(4) = 2$.
Theorem~\ref{thm:main-low-dims}\ref{d-leq-4} is based on the peculiarity that in dimensions at most four \emph{every} relaxation of~$X$ is bounded, whereas for dimensions $d \geq 5$ this is not necessarily the case.
We obtain this result using a description of maximal lattice-free sets provided by Lov\'asz~\cite{lovasz1989geometry}.
Once the boundedness of every relaxation is established, it is not hard to deduce that a relaxation $P$ of~$X$ having $k$ facets can be modified to a rational relaxation of $X$ that still has~$k$ facets by slightly moving the facets out and perturbing the facet normals to rational normals.
The details will be discussed in Section~\ref{sect:small-dimensions}.

The proof of Theorem~\ref{thm:main-low-dims}\ref{d=3} relies on Theorem~\ref{thm:main-conds}, so we discuss this first.
To prove Theorem~\ref{thm:main-conds}, we investigate the structure of the set of so-called observers of~$X$. We say that a point $y \in \Z^d  \setminus X$ outside a lattice-convex set $X \subseteq \Z^d$ is an \emph{observer} of~$X$ if the set $X \cup \{y\}$ is lattice-convex as well.

It turns out that whenever~$P$ is a (rational) polyhedron that contains~$X$ as a subset, but that does not contain any observer of $X$, then $P$ is actually a relaxation of~$X$.
Moreover, if~$X$ satisfies any of the three conditions in Theorem~\ref{thm:main-conds}, then the set of observers of~$X$ is finite and can be computed.
It thus remains to determine the minimum number~$k$ of inequalities that is sufficient to separate~$X$ from its observers.
The latter task can be carried out algorithmically using mixed-integer linear programming as an auxiliary tool.
This approach is developed in Section~\ref{sect:non-hollow}.

For proving Theorem~\ref{thm:main-low-dims}\ref{d=3} it suffices to consider sets~$X$ of dimension $3$, as for $\dim(X) \le 2$ we can use the computability of the relaxation complexity in dimension two established by Weltge~\cite{weltge2015diss} in his thesis. 
As a first step towards computability of $\rc(X)$ in dimension three, we characterize those~$X$ that have finitely many observers, and deal with them as in Theorem~\ref{thm:main-conds} as discussed before.
If the set of observers of~$X$ is infinite, then it still turns out to be structured enough for an algorithmic treatment.
To this end, we need to solve a special quantifier elimination problem for mixed-integer linear quantified expressions, which we find interesting in its own right.
These arguments will be laid out in Section~\ref{sect:d=3}.

\smallskip
\paragraph*{\textbf{Basic notation and terminology.}} The affine and convex hull of a set $X \subseteq \R^d$ are denoted by $\aff(X)$ and $\conv(X)$, respectively.
The line segment with endpoints $a,b \in \R^d$ is written as $[a,b] = \conv(\{a,b\})$.
By $\dim(X)$ we denote the dimension of $X$, which we define to be the dimension of the affine hull of~$X$.
For a positive integer~$m$, we write $[m] = \{1,\ldots,m\}$.
We use standard terminology from polyhedral theory such as polyhedron, vertex, face and facet and basic notions of the geometry of numbers, such as lattice. A lattice point is a point of the integer lattice $\Z^d$. We define an affine lattice to be a translation of a lattice by an arbitrary translation vector. By $e_1,\ldots,e_d$ we denote the standard unit vectors of~$\R^d$.
Two sets $A,B \subseteq \R^d$ are called unimodularly equivalent, if there is an affine unimodular transformation $f : \R^d \to \R^d$ such that $A = f(B)$, and where $f(x) = Ux + t$, with $U$ being an integral $(d \times d)$-matrix of determinant $\pm 1$, and $t \in \Z^d$.
Elements of $\R^d$ are interpreted as columns in analytic expressions.
For background information on these concepts we refer to the textbooks~\cite{schrijver1986theory} and~\cite{gruber2007convex}.

\section{Lower bounds on \texorpdfstring{$\rc(X)$}{rc(X)} in terms of the dimension of \texorpdfstring{$X$}{X}}
\label{sect:rc-vs-dim}

In this section, we prove Theorem~\ref{thm:rc-vs-dim-intro}. Although our result is still far from answering~\ref{q1}, it is the best lower bound known so far.
Our argument is inspired by the proof and the result of Weltge~\cite[Prop.~8.1.4]{weltge2015diss} who established the implication
\[
	\dim(X) \ge k! \qquad \Longrightarrow \qquad \rc(X)  \ge k
\]
for $X = \{0,e_1,\ldots,e_d\}$ by induction on $k$.
In fact, Weltge's argument can be applied for an arbitrary finite lattice-convex set $X$. 
We are able to replace~$k!$ by a single-exponential function in~$k$ by replacing his inductive argument with a pigeonhole-principle type argument. As in the proof of Weltge, we need the following auxiliary result.

\begin{lemma}[{\cite[Lem.~4]{averkov2013lovasz}}]
	\label{lem:unb:nonhollow}
	Every unbounded full-dimensional polyhedron $P \subseteq \R^d$ that contains a lattice point in its interior, contains infinitely many lattice points in its interior. 
\end{lemma}

\begin{theorem}
	\label{thm:dim:dependence}
	Let $X \subseteq \Z^d$ be a finite lattice-convex set satisfying 
	\[
		\dim(X) \ge (k-1) \binom{k}{\floor{k/2}}
	\]
	for some integer value $k \ge 2$. Then $\rc(X) \ge k+1$. 
\end{theorem}
\begin{proof}
	By restricting considerations to the affine hull of $X$,
	without loss of generality, we can assume that $X$ is full-dimensional and thus $d = \dim(X) \ge (k-1) \binom{k}{\floor{k/2}}$. Assuming that $X$ has a relaxation 
	\[
		P = \setcond{x \in \R^d}{a_1^\intercal x \le \beta_1,\ldots,a_k^\intercal x \le \beta_k}
	\]
	given by $k$ linear inequalities, we derive a contradiction.
	To this end, fix $p_1,\ldots,p_{d+1}$ to be $d+1$ affinely independent points in $X$.  With each $j \in [d+1]$ we associate the set 
	\[
		I_j:= \setcond{i \in [k]}{a_i^\intercal p_j = \beta_i}
	\] of indices of the inequalities of the relaxation of $X$ that are active on~$p_j$. None of the sets $I_1,\ldots,I_{d+1}$ is empty. Indeed, the relaxation~$P$ is unbounded, because since $k\geq2$, we have $k \le (k-1) \binom{k}{\floor{k/2}} \le d$. Thus, if~$I_j$ were empty, then $p_j$ would be an interior lattice point of $P$. By Lemma~\ref{lem:unb:nonhollow} it then follows that~$P$ contains infinitely many lattice points, a contradiction. 
	
	To explain the proof idea, we first show a weaker assertion, namely that
	\begin{align}
	d = \dim(X) \ge (k-1) 2^k \qquad \Longrightarrow \qquad \rc(X)  \ge k+1.\label{eqn:weaker-rc-vs-dim-bound}
	\end{align}
The inequality $\dim(X) \ge (k-1) 2^k$ implies that $I_1,\ldots,I_{d+1}$ is a list of at least $(k-1) 2^k + 1$ subsets of~$[k]$. It follows that there exists a subset $I \subseteq [k]$ that occurs in the list $I_1,\ldots,I_{d+1}$ at least $\ceil{(d+1 )/ 2^k}$ times, which means that the set
	\[
		J := \setcond{ j \in [d+1] }{I_j = I}
	\]
	has at least $\ceil{(d+1)/ 2^k}$ elements. We have 
	\[
		\frac{d+1}{2^k} \ge \frac{(k-1) 2^k + 1}{2^k} > k-1,
	\]
	so that $\card{J} \geq k$. Without loss of generality assume that $I_1 = \ldots = I_k = I$. Let $H := \aff(\{p_1,\ldots,p_k\})$. The inequality $a_j^\intercal x \le b_j$ holds with equality for all $x \in H$ and all $j \in I$. Hence, the polyhedron
	\[
		Q = \setcond{x \in H}{a_j^\intercal x \le b_j \ \text{for all} \ j \in [k] \setminus I}
	\]
	is a relaxation of the $(k-1)$-dimensional set $H \cap X$ within the affine lattice $H \cap \Z^d$. The relaxation is given by $k - |I| \le k-1$ inequalities. Since $\dim(H) = k-1$, the relaxation $Q$ is unbounded. By construction, the $k - |I|$ inequalities defining the relaxation $Q$ hold strictly on $p_1,\ldots,p_k$. Thus, $p_1,\ldots,p_k$ belong to the relative interior of $Q$, which contradicts Lemma~\ref{lem:unb:nonhollow}.
	
	For the improved assertion, the approach is similar, but we use Sperner's theorem on the size of antichains in the boolean lattice to strengthen the argument. The idea is that we do not need to have all of the points $p_1,\ldots,p_k$ in the relative interior of $Q$, but rather just one of them, in order to obtain a contradiction. 
	
	To this end, observe that the family of sets $\{I_1,\ldots,I_{d+1}\}$ is partially ordered by inclusion, and consider the $m$-element subfamily $\{S_1,\ldots,S_m\}$ of all inclusion-minimal elements of $\{I_1,\ldots,I_{d+1}\}$.
	This subfamily consists of pairwise incomparable subsets of~$[k]$	and thus forms an antichain in the boolean lattice of subsets of~$[k]$.
	Sperner's theorem~\cite{sperner1928einsatz} asserts that $m \le \binom{k}{\floor{k/2}}$. As each of the sets $I_1,\ldots,I_{d+1}$ contains one of the inclusion-minimal sets $S_1,\ldots,S_m$ we have
	\[
		[d+1] = \bigcup_{t=1}^m \setcond{j \in [d+1]}{S_t \subseteq I_j}.
	\]
	Consequently, 
	\begin{align*}
		d+1  & = \left| \bigcup_{t=1}^m \setcond{j \in [d+1]}{S_t \subseteq I_j} \right| 
		\\  & \le \ \sum_{t=1}^m \left| \setcond{j \in [d+1]}{S_t \subseteq I_j} \right|
		\\ & \le m \cdot \max_{t \in [m]} \left| \setcond{j \in [d+1]}{S_t \subseteq I_j} \right|.
	\end{align*}
	We have thus shown the existence of an index $t \in [m]$ such that $S_t$ is contained in at least $\ceil{(d+1)/m}$ sets from the list $I_1,\ldots,I_{d+1}$. Without loss of generality, we assume that this holds for the set $S_1$. 
Using the lower bound on $\dim(X)$, we see that 
\[
\frac{d+1}{m} \ge \frac{(k-1) \binom{k}{\floor{k/2}}+1}{\binom{k}{\floor{k/2}}}  > k - 1.
\]
Thus, at least $k$ of the sets from the list $I_1,\ldots,I_{d+1}$ contain $S_1$ as a subset, and we may assume that $S_1 \subseteq I_j$, for $j \in [k]$, and that $S_1 = I_1$.
We can now repeat the above argument, replacing the set $I$ by~$S_1$, in order to find that the point $p_1$ lies in the relative interior of the corresponding relaxation~$Q$ of $H \cap X$ within $H \cap \Z^d$.
%
\end{proof}


Let us derive an explicit lower bound on $\rc(X)$ in terms of $\dim(X)$, which goes to infinity, when $\dim(X) \to \infty$.

\begin{proof}[Proof of Theorem~\ref{thm:rc-vs-dim-intro}]
Recall that we want to prove that 
\[
	\rc(X) > \log_2 (\dim(X)) - \log_2 \log_2 (\dim(X)).
\]
In fact, we will see that this already follows from the weaker statement~\eqref{eqn:weaker-rc-vs-dim-bound} which does not rely on Sperner's Theorem.
So, let $k \geq 2$ be maximal such that $\dim(X) \geq (k-1) 2^k$.
Then, by~\eqref{eqn:weaker-rc-vs-dim-bound} and since $\dim(X) < k 2^{k+1}$, we get
\[
\rc(X) \geq k+1 > \log_2(\dim(X)) - \log_2(k).
\]
The claimed inequality follows, since $k \geq 2$ implies $\dim(X) \geq (k-1) 2^k \geq 2^k$, and thus $\log_2(k) \leq \log_2 \log_2 (\dim(X))$.
\end{proof}

\section{The role of rationality in dimensions \texorpdfstring{$d \leq 4$}{d<=4}}
\label{sect:small-dimensions}

This part is mainly devoted to proving Theorem~\ref{thm:main-low-dims}\ref{d-leq-4}, that is, showing that $\rc(X) = \rc_\Q(X)$, for all at most four-dimensional lattice-convex sets~$X$.
We also see how the developed methods enable us to answer~\ref{q1} affirmatively in these dimensions.
Our main observation is that there is a qualitative difference between low and high dimensions: We show that in dimensions up to four, relaxations of finite sets are necessarily bounded, while in higher dimensions this is not necessarily the case.

The arguments are based on the notion of \emph{maximal lattice-free} sets.
We call a $k$-dimensional closed convex set $L \subseteq \R^d$ such that $\aff(L) \cap \Z^d$ is a $k$-dimensional affine lattice, a $k$-dimensional \emph{lattice-free} set if the relative interior of $L$ does not contain points of $\Z^d$. Further, we call such a set~$L$ a $k$-dimensional \emph{maximal} lattice-free set, if $L$ is not properly contained in another $k$-dimensional lattice-free set.

\begin{proposition}[{}]
	\label{prop:maximizing}
	Every $k$-dimensional lattice-free set is a subset of a maximal $k$-dimensional lattice-free set. 
\end{proposition}
\begin{proof}
	This is well-known and can be easily derived from Zorn's lemma, or by a topological argument~\cite[Prop.~3.1]{averkovwagner2012inequ}; see also Basu et al.~\cite[Cor.~2.2]{basuconforticornuejolszambelli2010maximal} for a more constructive proof.
\end{proof}

The following structural result has been formulated by Lov\'asz in~\cite[Sect.~3]{lovasz1989geometry}. A complete proof can be found in~\cite{averkov2013lovasz} and~\cite{basuconforticornuejolszambelli2010maximal}.

\begin{theorem}
	\label{thm:lovasz}
	Every $d$-dimensional maximal lattice-free set $L$ is a polyhedron. If $L$ is bounded, then~$L$ has at most $2^d$ facets and the relative interior of each facet contains a point of the lattice~$\Z^d$.
	If $L$ is unbounded, then up to unimodular transformations, $L$ is equivalent to $L' \times \R^m$ for some $m \in \{1,\ldots,d-1\}$ and some bounded $(m-d)$-dimensional maximal lattice-free set $L' \subseteq \R^{d-m}$.
\end{theorem}

The main insight towards the aforementioned results is to show that relaxations of low-dimensional finite lattice-convex sets are always bounded.
We say that a vector $r \in \R^d$ is a \emph{recession vector} of a polyhedron $P \subseteq \R^d$ if $P + r \subseteq P$, the ray $\gamma = \setcond{\lambda \, r}{\lambda \geq 0}$ in direction of~$r$ is called a \emph{recession ray}.
The set of all recession vectors of $P$ is called the \emph{recession cone} $\rec(P)$ of~$P$.

\begin{lemma}
	\label{lem:bounded:relaxation}
	Let $X \subseteq \Z^d$ be a finite lattice-convex set and let one of the following conditions hold:
	\begin{enumerate}[label=\arabic*.]
		\item $d=2$, $\dim(X) \ge 1$,
		\item $d=3$, $\dim(X) \ge 2$, 
		\item $d=4$, $\dim(X) = 4$.
	\end{enumerate}
	Then, every relaxation $P$ of $X$ is bounded.
\end{lemma}
\begin{proof}
	Let the polyhedron $P$ be a relaxation of $X$, which means $P \cap \Z^d = X$. We assume to the contrary that $P$ is unbounded, that is, by the basic theory of convexity there is a recession ray $\gamma$ of~$P$, with recession vector $r \in \rec(P) \setminus \{0\}$ (cf.~\cite{schrijver1986theory}).
	
	\emph{Case~1: $d=2, \dim(X) \ge 1$.} We consider two distinct points $a,b$ of $X$. If the ray $\gamma$ is parallel to $[a,b]$, say $r = b-a$, then  $a + k (b-a)$, with $k \in \Z_{\ge 0}$, are infinitely many lattice points that are contained in $P$, which is a contradiction. If $\gamma$ is not parallel to $[a,b]$, consider the two-dimensional set $[a,b]+\gamma \subseteq P$. If $[a,b]+\gamma$ is not lattice-free, then Lemma~\ref{lem:unb:nonhollow} yields that $[a,b]+\gamma$ contains infinitely many interior lattice points, which contradicts $P \cap \Z^2 = X$. If $[a,b]+\gamma$ is lattice-free, then by Proposition~\ref{prop:maximizing} there is a two-dimensional maximal lattice-free set $L$ containing $[a,b]+\gamma$. In view of Theorem~\ref{thm:lovasz}, $L$ is unimodularly equivalent to $[0,1] \times \R$. Thus, the recession cone of $L$ is a rational line. On the other hand, the recession cone of $P$ contains $\gamma$ as a subset. Hence, the ray $\gamma$ has a rational direction. It follows that $a + \gamma$ contains infinitely many points of $\Z^2$, which again contradicts $P \cap \Z^2 = X$. 
	
	\emph{Case~2: $d=3, \dim(X) \ge 2$.} We consider three affinely independent points $a,b,c \in X$ and the triangle $T= \conv(\{a,b,c\})$. If $\gamma$ is parallel to the plane affinely spanned by $a,b,c$, then using Case~1 for the subset $\aff(\{a,b,c\}) \cap X$ of the two-dimensional affine lattice $\aff(\{a,b,c\}) \cap \Z^3$, we arrive at a contradiction. Otherwise, $T + \gamma$ is a three-dimensional subset of $P$. If $T + \gamma$ is not lattice-free, then Lemma~\ref{lem:unb:nonhollow} yields that $T + \gamma$ contains infinitely many lattice points, which is a contradiction. If $T + \gamma$ is lattice-free, then let $L$ be a three-dimensional maximal lattice-free set containing $T + \gamma$. By Theorem~\ref{thm:lovasz}, $L$ is unimodularly equivalent to a set of the form $L' \times \R^m$, where $m \in \{1,2\}$ and $L'$ is a bounded $(3-m)$-dimensional maximal lattice-free set.
	
	If $m=1$, then the recession cone of $L$ is a rational line and so $\gamma$, being a subset of the recession cone of $P$ is a rational ray. This shows that $a + \gamma$ contains infinitely many points of $\Z^3$, again a contradiction.
	
	If $m=2$, then the boundary of $L$ is a union of two parallel planes and one of these planes, which we denote by $H$, contains at least two of the points~$a,b,c$. The ray $\gamma$ is parallel to $H$, and $H \cap X$ is a lattice-convex set of dimension at least one in the affine lattice $H \cap \Z^3$. Applying the assertion of Case~1 to the set $H \cap X$ in $H \cap \Z^3$, yields the desired contradiction. 
	
	\emph{Case~3: $d=4, \dim(X) = 4$.} We pick five affinely independent points $p_0,\ldots,p_4$ in $X$ and consider the simplex $S = \conv(\{p_0,\ldots,p_4\})$. Clearly, $S + \gamma$ is $4$-dimensional. If $S + \gamma$ is not lattice-free, we get a contradiction just as in the previous cases. If $S+ \gamma$ is lattice-free, again analogously to the previous cases, we find a maximal $4$-dimensional lattice-free set $L \supseteq S + \gamma$. This set is unimodularly equivalent to $L' \times \R^m$, where $m \in \{1,2,3\}$ and $L'$ is a bounded $(4-m)$-dimensional maximal lattice-free set. 
Without loss of generality, we assume that $L = L' \times \R^m$.
		
	If $m=1$, then the recession cone of $L$ is a rational line, and so $\gamma$ is a ray in a rational direction. In this case, $p_0+\gamma$ contains infinitely many points of $\Z^4$, which is a contradiction. 
	
	If $m=2$, then $L'$ is a bounded two-dimensional maximal lattice-free set. Consider the projection map $\pi : \R^4 \to \R^2$, $\pi(x_1,x_2,x_3,x_4) = (x_1,x_2)$. If two of the points $\pi(p_0),\ldots,\pi(p_4)$ coincide, say $q = \pi(p_0) = \pi(p_1)$, then the fiber $H:=\pi^{-1}(q)$ is a two-dimensional affine space containing the points $p_0$ and $p_1$. Thus, we can use Case~1 for the set $X \cap H$ of dimension at least one in the two-dimensional affine lattice $H \cap \Z^4$ to arrive at a contradiction. Thus, we can assume that the five points $q_i = \pi(p_i)$, with $i \in \{0,1,\ldots,4\}$, are pairwise distinct lattice points in~$L'$. 
	
	Maximal lattice-free sets in dimension two are completely classified (see~\cite[Thm.~2]{averkovkruempelmannweltge2017notions}). The classification restricts~$L'$ as follows:
	\begin{enumerate}[label=(\roman*)]
		\item $L'$ is a triangle or a quadrilateral.
		\item If $L'$ is a quadrilateral, then $L'$ contains exactly four lattice points.
		\item If $L'$ is a triangle, then all but two lattice points of $L'$ are contained in the same edge, or $L'$ is unimodularly equivalent to $\conv(\{0, 2 e_1, 2 e_2\})$.
	\end{enumerate}
	Since $q_0,\ldots,q_4$ are five distinct points, $L'$ cannot be a quadrilateral. Thus, $L'$ is a triangle and we conclude that three of these five points, say $q_0,q_1,q_2$, lie in the same edge of $L'$. Then $p_0,p_1,p_2$ lie in the same facet of $L$. Let us denote by $H$ the hyperplane in~$\R^4$ spanned by the facet of $L$ that contains $p_0,p_1,p_2$. The set $H \cap X$ is a lattice-convex set of dimension at least two in the three-dimensional affine lattice $H \cap \Z^4$. Thus, we can use assertion of Case~2 to arrive at a contradiction. 

		If $m=3$, then the boundary of $L$ is a union of two parallel planes and one of these planes, which we denote by $H$, contains at least three of the affinely independent points~$p_0,\ldots,p_4$. The ray $\gamma$ is parallel to $H$, and $H \cap X$ is a lattice-convex set of dimension at least two in the three-dimensional affine lattice $H \cap \Z^4$. Applying the assertion of Case~2 to the set $H \cap X$ in $H \cap \Z^4$, yields the desired contradiction. 
\end{proof}

Lemma~\ref{lem:bounded:relaxation} is naturally constrained to small dimensions.
The following five-dimensional example with an unbounded relaxation is taken from Kaibel \& Weltge~\cite[Ex.~1]{kaibelweltge2015lowerbounds} (cf.~\cite[Sect.~7.3]{weltge2015diss}).

\begin{example}
\label{ex:five-dim-simplex}
Let $X = \{0,e_1,e_2,e_3,e_1+e_3+e_4,e_2+e_3+e_5\} \subseteq \Z^5$ and let $r = (0,0,0,1,\sqrt{2})^\intercal$.
Then, the unbounded polyhedron $\conv(X) + \R r$ is a relaxation of $X$.
\end{example}

Moreover, this example can be used to construct arbitrarily large lattice-convex sets in every dimension $d \geq 6$ admitting an unbounded relaxation.
Indeed, if $X \subseteq \Z^d$ is a finite lattice-convex set, and $u \in \R^d \setminus \{0\}$ is a direction such that $(\conv(X) + \R u) \cap \Z^d = X$, then for every $k \in \N$, the cartesian product $\tilde X = X \times \{0,1,\ldots,k\}\subseteq \Z^{d+1}$ is a lattice-convex set admitting the unbounded relaxation $\conv(\tilde X) + \R \binom{u}{0}$.

Based on Lemma~\ref{lem:bounded:relaxation}, we can now show that rationality does not play a distinguished role in dimensions at most four, and thus prove Theorem~\ref{thm:main-low-dims}\ref{d-leq-4}.

\begin{theorem}
	\label{thm:rc:gen:vs:rat:d<=4}
	Let $d \le 4$. Then, $\rc(X) = \rc_\Q(X)$ for every finite lattice-convex set $X \subseteq \Z^d$. 
\end{theorem}
\begin{proof}
	For computing $\rc(X)$ and $\rc_\Q(X)$, we can pass to the affine lattice $\aff(X) \cap \Z^d$. Thus, without loss of generality we can assume that $X$ is a $d$-dimensional lattice-convex set in~$\Z^d$. The inequality $\rc(X) \le \rc_\Q(X)$ is trivial and so we need to show $\rc_\Q(X) \le \rc(X)$. 
	
	Lemma~\ref{lem:bounded:relaxation} implies that, under our assumptions, every relaxation $P$ of $X$ is bounded. Choose a relaxation $P$ with $k:=\rc(X)$ facets. It suffices to prove the existence of a rational relaxation with at most~$k$ facets. First note that by slightly increasing the right hand sides of the inequality description of $P$, we can assume that $X$ is contained in the interior of~$P$. Let $F_1,\ldots,F_k$ be the facets of $P$. Each of these facets is disjoint with $X$ and so, for each $i \in \{1,\ldots,k\}$, there is a hyperplane $H_i$ that separates $X$ from $F_i$, which means that $H_i$ determines halfspaces $H_i^+$ and $H_i^-$ such that $X$ lies in the interior of $H_i^+$ and $F_i$ lies in the interior of $H_i^-$. Since~$P$ is bounded, $H_i$ can be chosen to be a rational hyperplane. It follows that $R := \bigcap_{i=1}^k H_i^+$ is a rational relaxation of $X$ satisfying $R \subseteq P$ and having at most $k$ facets. This shows $\rc_\Q(X) \le \rc(X)$. 
\end{proof}

\begin{remark}
The proof of Theorem~\ref{thm:rc:gen:vs:rat:d<=4} works for every $X \subseteq \Z^d$ such that every of its relaxations is bounded, independently of the dimension~$d$.
\end{remark}

For the sake of a discussion of~\ref{q1} in small dimensions, observe that by Theorem~\ref{thm:dim:dependence} a lattice-convex set $X \subseteq \Z^d$ is guaranteed to satisfy $\rc(X) \geq k+1$, for $k=3$ and $k=4$, only if $\dim(X) \geq 6$ and $\dim(X) \geq 18$, respectively.
We solve~\ref{q1} in these cases, by providing optimal bounds.

\begin{corollary}
	\label{cor:rc>=d+1:d<=4}
	Let $d \le 4$.
	Then, $\rc(X) \ge d+1$ holds for every $d$-dimensional finite lattice-convex set $X \subseteq \Z^d$.
\end{corollary}
\begin{proof}
	By Lemma~\ref{lem:bounded:relaxation}, every relaxation of $X$ is bounded. Every bounded $d$-dimensional polyhedron has at least $d+1$ facets. This gives $\rc(X) \ge d+1$.
\end{proof}

As a further consequence, we get that $\rc(\Delta_d)=d+1$, for every $d \leq 4$, where $\Delta_d = \{0,e_1,\ldots,e_d\}$ is the set of lattice points of the standard simplex.
Weltge~\cite[Prob.~11]{weltge2015diss} (cf.~\cite{kaibelweltge2015lowerbounds}) conjectures that this identity holds in arbitrary dimension.
However, even for this particular case we need to develop new tools, because the simplex in Example~\ref{ex:five-dim-simplex} is a unimodular image of $\conv(\Delta_5)$, so in particular there are unbounded relaxations of~$\Delta_5$.

\section{Conditions on a lattice-convex set to have finitely many observers}
\label{sect:non-hollow}
This section is concerned with the proof of Theorem~\ref{thm:main-conds}.
Our argument is split up into two main parts:
First, we study the set of so-called observers of a lattice-convex set~$X \subseteq \Z^d$, which is a subset of the lattice points outside of~$X$ that is of course necessary, but more importantly, also sufficient to be separated by the minimal number of inequalities.
By applying techniques from the Geometry of Numbers we find that the set of observers of~$X$ is finite if (a) $X$ is parity-complete, (b) $\conv(X)$ is not lattice-free, or (c) the lattice-width of $X$ is not too small.
We introduce these notions below.

In the second part, we explain how mixed-integer linear programming (MILP) can be used to compute the minimal number of inequalities that are needed to separate~$X$ from a \emph{finite} subset $Y \subseteq \Z^d \setminus X$.
We moreover argue that for the separation problem for such finite sets $X$ and~$Y$, there is no loss of generality to restrict to \emph{rational} linear descriptions.

\subsection{The set of observers of a lattice-convex set}
\label{sect:observers}

\begin{definition}
\label{def:observers}
Let $X \subseteq \Z^d$ be a finite lattice-convex set.
We say that a point $y \in \Z^d \setminus X$ \emph{observes} $X$ if $\conv(X \cup \{y\}) \cap \Z^d = X \cup \{y\}$, that is, $X \cup \{y\}$ is lattice-convex as well.
Write
\[
\obs(X) := \setcond{y \in \Z^d \setminus X}{y\text{ observes }X}
\]
for the set of points that observe $X$.
\end{definition}
\begin{figure}[ht]
\hfill\includegraphics[scale=.8,page=1]{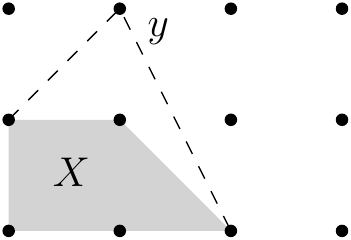}
\hfill\includegraphics[scale=.8,page=2]{observer-example.pdf}
\hfill\,
\caption{A lattice-convex set $X$ with an observer $y$ and a non-observer $z$.}
\label{fig:observer-example}
\end{figure}
Our notion of observers is inspired by Weltge's definition of a \emph{guard set for~$X$}, which is a set $G \subseteq \Z^d \setminus X$ with the property that for every $p \in \Z^d \setminus X$ we have $G \cap \conv(X \cup \{p\}) \neq \emptyset$.
Indeed, every guard set~$G$ contains $\obs(X)$, so that the set of observers is the smallest guard set with respect to inclusion.
Weltge proved that for two-dimensional lattice-convex sets $X \subseteq \Z^2$ there is always a finite guard set, and thus in particular there are only finitely many observers.

\begin{proposition}[{Weltge~\cite[Prop.~7.5.6 \& Thm.~7.5.7]{weltge2015diss}}]\label{prop:finite-observers-dim-2}
If $X \subseteq \Z^2$ is a full-dimensional finite lattice-convex set, then $\obs(X)$ is finite and can be computed algorithmically.
\end{proposition}

However, if $X \subseteq \Z^d$ is such that $\dim(X) < d$, then $\card{\obs(X)} = \infty$.
Indeed, every lattice point in a neighboring lattice plane to $\aff(X)$ is an observer of~$X$.
Even more, in dimensions $d \geq 3$, there are full-dimensional lattice-convex sets that have infinitely many observers.
One example is the set $\Delta_d = \{0,e_1,\ldots,e_d\}$ of lattice points of the standard simplex.

For the sake of convenient notation, we extend the definition of the relaxation complexity as follows:
For $X \subseteq Y \subseteq \Z^d$, the minimal $k$ such that $X$ can be separated from~$Y \setminus X$ by~$k$ linear inequalities (resp. rational linear inequalities) is denoted by $\rc(X,Y)$ (resp.~$\rc_\Q(X,Y)$).
So, in particular $\rc(X) = \rc(X,\Z^d)$ and $\rc_\Q(X)=\rc_\Q(X,\Z^d)$.

The utility of the concept of observers stems from the fact that for a polyhedron $P$ to be a relaxation of~$X$, it suffices that $P$ separates~$X$ from $\obs(X)$.
This follows directly from the definition of~$\obs(X)$.

\begin{proposition}
	\label{prop:obs-observations:XuObs}
	Let $X \subseteq \Z^d$ be a finite lattice-convex set and let $A x \le b$ be a system of linear inequalities.
	The following conditions are equivalent: 
	\begin{enumerate}[label=(\roman*)]
		\item The system $A x \le b$ separates $X$ from $\Z^d \setminus X$.
		\item The system $A x \le b$ separates $X$ from $\obs(X)$. 
	\end{enumerate}
In particular,
\[
\rc(X) = \rc(X, X \cup \obs(X)) \quad \text{and} \quad \rc_\Q(X) = \rc_\Q(X, X \cup \obs(X)).
\]
\end{proposition}

\noindent If $\obs(X)$ is finite and can be computed, then it serves as a finite certificate for $\rc(X)$ and~$\rc_\Q(X)$, that allows to algorithmically determine the minimal size relaxation of~$X$.
Before we develop this algorithm in Section~\ref{sect:sep-milp}, we derive three sufficient conditions on a lattice-convex set $X \subseteq \Z^d$ under which there are only finitely many observers.

\subsubsection{Parity-complete sets}

For every fundamental cell~$F$ of~$\Z^d$ each residue class in $(\Z / 2\Z)^d$ has a representative that is a vertex of~$F$.
A generic lattice-convex set $X \subseteq \Z^d$ with sufficiently many points will contain a fundamental cell of~$\Z^d$.
These observations motivate the following class of examples and show its abundance.

We call $X \subseteq \Z^d$ \emph{parity-complete} if for every lattice point $z$ in the affine hull of $X$ there exists an $x \in X$ congruent to $z$ modulo $2$, which means that $(x-z)/2 \in \Z^d$. 

\begin{theorem}
\label{thm:parity-complete}
Let $X \subseteq \Z^d$ be a full-dimensional finite parity-complete and lattice-convex set.
Then, $\obs(X) \subseteq 2 X - X$ and in particular $\obs(X)$ is finite and computable.
\end{theorem}

\begin{proof}
Let $z \in \Z^d$ be an observer of $X$.
Since $X$ is parity-complete, there exists $x \in X$ satisfying $(x+z)/2 \in \Z^d$.
We conclude that $x':=(x+z) / 2 \in X$, since otherwise $X$ would not be an observer.
We thus have $z = 2 x' - x \in 2 X - X$.
\end{proof}

\subsubsection{Existence of interior lattice points}

A second class of lattice-convex sets with only finitely many observers is given by those $X$ for which $\conv(X)$ is not lattice-free.
Before we can prove this result we need to revisit some crucial results in the Geometry of Numbers.

We call the convex hull of finitely many lattice points a \emph{lattice polytope}, as usual.
Blichfeldt's theorem~\cite{blichfeldt1914anew} is a classical upper bound on the number of lattice points in a full-dimensional lattice polytope $P \subseteq \R^d$ in terms of its volume.
It states that
\begin{align}
\card{P \cap \Z^d} &\leq d! \vol(P) + d. \label{eqn:blichfeldt}
\end{align}
A lower bound on $\card{P \cap \Z^d}$ holds if $P$ is not lattice-free, that is, its set $\inter(P) \cap \Z^d$ of interior lattice points is non-empty.
However, the best possible such bound is still not known and this problem received a considerable amount of interest in the last years.
The best result to date is due to~\cite{averkovkruempelmannnill2018lattice} and reads
\begin{align}
\vol(P) &\leq c_d^d\,  \card{\inter(P) \cap \Z^d}, \label{eqn:best-hensley}
\end{align}
where $c_d = d(2d+1)(s_{2d+1}-1)$ and $(s_i)_{i \in \Z_{>0}}$ is the Sylvester sequence.
This sequence is recursively defined by $s_1 = 2$ and $s_i = 1 + s_1\cdot\ldots\cdot s_{i-1}$, for $i \geq 2$.
It grows double-exponentially and satisfies the upper bound $s_i \leq 2^{2^{i-1}}$.

The proof of Inequality~\eqref{eqn:best-hensley} is based on estimating the \emph{coefficient of asymmetry} $\ca(P,y)$ of the polytope $P$ with respect to an interior point $y \in \inter(P)$.
This magnitude is defined as
\[
\ca(P,y) := \max_{\|x\|=1} \frac{\max\setcond{\lambda > 0}{y + \lambda x \in P}}{\max\setcond{\lambda > 0}{y - \lambda x \in P}}.
\]
In~\cite[Thm.~1.4]{averkovkruempelmannnill2018lattice} it is proven that there exists a lattice point $y \in \inter(P) \cap \Z^d$ such that
\begin{align}
\ca(P,y) \leq c_d = d(2d+1)(s_{2d+1}-1).\label{eqn:ca-bound}
\end{align}

With these preparations we can now formulate and prove our anticipated description of the set of observers of a lattice-convex set $X$ with the property that $\conv(X)$ is not lattice-free.
Our arguments are somewhat similar to those used in~\cite[Thm.~12]{averkovconfortidelpiadisummafaenza2013onthe}.

\begin{theorem}
\label{thm:non-hollow-finite-obs}
Let $X \subseteq \Z^d$ be a finite lattice-convex set such that $\conv(X)$ is not lattice-free.
Then,
\[
\obs(X) \subseteq X + c_d \cdot \conv(X-X),
\]
where $c_d = d(2d+1)(s_{2d+1}-1)$.

In particular, $\obs(X)$ is finite and we have the explicit bound
\[
\card{\obs(X)} \leq c'_d \cdot \card{X},
\]
where $c'_d = d! \, (1+c_d)^{2d} \, \tbinom{2d}{d} \in \cO(d^{5d} 2^{3d \cdot 2^{2d}})$.
\end{theorem}

\begin{proof}
Let $p \in \obs(X)$ and write $Q_p = \conv(X \cup \{p\})$.
Further, let $y \in \inter(Q_p) \cap \Z^d$ be a lattice point satisfying~\eqref{eqn:ca-bound}, that is, $\ca(Q_p,y) \leq d(2d+1)(s_{2d+1}-1)$.
Since $p$ was taken to be an observer, we necessarily have that $y \in X$.
Let $z \in \conv(X)$ be the intersection point of $\conv(X)$ with the ray in direction $y-p$ and emanating from~$p$.
Then, by the definition of $\ca(Q_p,y)$, there is a positive number $0 < \lambda \leq \ca(Q_p,y)$ such that $p-y = \lambda (y-z)$.
Thus,
\[
p = y + \lambda (y-z) \in X + \lambda \cdot \conv(X-X) \subseteq X + \ca(Q_p,y) \cdot \conv(X-X),
\]
so that by~\eqref{eqn:ca-bound} the claimed inclusion holds with $c_d = d(2d+1)(s_{2d+1}-1)$.

In order to estimate the number of observers of~$X$, we first assume without loss of generality that $0 \in X$.
This implies $X \subseteq X - X$, which in turn gives $X + c_d \cdot \conv(X-X) \subseteq (1+c_d) \cdot \conv(X-X)$.
Blichfeldt's bound~\eqref{eqn:blichfeldt} applied to $P = \conv(X-X)$ gives us
\begin{align}
\card{\obs(X)} &\leq \card{((1+c_d) \cdot \conv(X-X)) \cap \Z^d} \nonumber\\
&\leq d! \, (1+c_d)^d \vol(\conv(X-X)) + d \nonumber\\
&\leq d! \, (1+c_d)^d \, \tbinom{2d}{d} \vol(\conv(X)) + d \label{eqn:RS}\\
& \leq d! \, (1+c_d)^d \, \tbinom{2d}{d} \, c_d^d \cdot \card{X} + d \ \leq \  c'_d \cdot \card{X}.\label{eqn:int-lat-points}
\end{align}
For the inequality~\eqref{eqn:RS} we observe that by $\conv(X-X) \subseteq \conv(X) - \conv(X)$ we can apply the Rogers-Shephard inequality~\cite{rogersshephard1957the}, whereas for~\eqref{eqn:int-lat-points} we use the fact that $X$ is lattice-convex, and employ the volume bound~\eqref{eqn:best-hensley}.
The claimed asymptotic growth of the dimensional constant $c'_d$ follows by that of $c_d$ and Stirling's approximation of~$d!$.
\end{proof}

\begin{remark}
\label{rem:pikhurkos-guess}
The best-known bound~\eqref{eqn:ca-bound} on the minimal coefficient of asymmetry of an interior lattice point~$y$ in a lattice polytope~$P \subseteq \R^d$ is certainly quite far from optimal.
Pikhurko~\cite{pikhurko2001lattice} proposes that the optimal bound should rather read
\[
\ca(P,y) \leq s_d^2-2,
\]
which would be an enormous improvement given the double-exponential growth of the Sylvester sequence.
\end{remark}

\subsubsection{Sets of large lattice-width}

Our third class of examples with only finitely many observers is informally described as those $X$ that cannot be sandwiched between two parallel lattice planes of small distance.
More precisely, for an integral vector $u \in \Z^d \setminus \{0\}$ the \emph{width} of a subset $S \subseteq \R^d$ in direction~$u$ is defined as
\[
w(S,u) = \max_{x \in S} u^\intercal x - \min_{x \in S} u^\intercal x,
\]
and the \emph{lattice-width} of~$S$ is defined as
\[
w(S) = \min_{u \in \Z^d \setminus \{0\}} w(S,u).
\]
In order to describe our result, we moreover need a concept introduced by Blanco et al.~\cite{blancohaasehofmannsantos2016finiteness}: The \emph{finiteness threshold width} is the constant $w^\infty(d) \in \Z_{>0}$ such that for every $n \in \Z_{>0}$, up to unimodular equivalence, all but finitely many lattice $d$-polytopes with $n$ lattice points have lattice-width at most $w^\infty(d)$.
In~\cite{blancohaasehofmannsantos2016finiteness} it is proven that $d-2 \leq w^\infty(d) \leq \cO(d^\frac32)$, and the authors obtain the exact values $w^\infty(3) = 1$ and $w^\infty(4) = 2$, the former being shown already in~\cite{blancosantos2016lattice}.

\begin{theorem}
\label{thm:hollow-threshold-width}
Let $X \subseteq \Z^d$ be a full-dimensional finite lattice-convex set.
If $w(X) > w^\infty(d)$, then $\obs(X)$ is finite.
\end{theorem}

\begin{proof}
Assume for contradiction that $\card{\obs(X)}=\infty$.
Let $p \in \obs(X)$ and write $Q_p = \conv(X \cup \{p\})$.
Observe that $\card{Q_p \cap \Z^d} = \card{X} + 1 =: n$ and $w(Q_p) \geq w(X) > w^\infty(d)$.
Now, for every $M \geq 0$, there is an observer $p \in \obs(X)$ outside the box $[-M,M]^d$.
Since $X$ is full-dimensional the corresponding polytope $Q_p$ has a facet $F$ such that the height of~$p$ over~$F$ is lower bounded by an increasing function in~$M$.
As a consequence there are infinitely many possible values for the volume of $Q_p$, and hence there are infinitely many unimodularly non-equivalent lattice polytopes $Q_p$, which all have exactly $n$ lattice points and lattice-width $> w^\infty(d)$.
This contradicts the definition of the finiteness threshold width.
\end{proof}

Unlike in Theorem~\ref{thm:parity-complete} and Theorem~\ref{thm:non-hollow-finite-obs}, the proof of Theorem~\ref{thm:hollow-threshold-width} does not provide a mean to algorithmically compute the set $\obs(X)$ in the case that $w(X) > w^\infty(d)$.
Computability would follow if we are given an explicit upper bound $f(w,d,n)$ on the volume of the finitely many lattice $d$-polytopes with $n$ lattice points and lattice-width $w > w^\infty(d)$.
To the best of our knowledge such an explicit volume bound has not been proven by the time of writing.
Moreover, it is not clear how to determine the constant $w^\infty(d)$ algorithmically in a given dimension~$d$.

However, in a later section we develop a general algorithm that computes $\obs(X)$ under the sole assumption that this set is finite (see Theorem~\ref{thm:finite-obs-computability}).

\subsection{Separation of two finite sets using MILP} 
\label{sect:sep-milp}

For sets $X, Y \subseteq \R^d$ we say that \emph{$X$ is separated from $Y$ by a system $Ax \le b$} of $k$ linear inequalities, if $A x\le b$ is fulfilled for every $x \in X$ and not fulfilled for any $y \in Y$. 
In this section, we address the following computational problem. 

\begin{framed}[.85\textwidth]
\centerline{\emph{Separation Problem}}
\smallskip
\noindent Given finite subsets $X$ and $Y$ of $\Q^d$ and $k \in \Z_{>0}$, determine a system $A x \le b$ of $k$ linear inequalities that separates $X$ from~$Y$. 
\end{framed}

\noindent Phrased geometrically, the problem asks to determine a polyhedron $P$ with at most $k$ facets satisfying $X \subseteq P$ and $P \cap Y = \emptyset$. 
In view of this interpretation it is clear that the separation problem is strongly related to the notion of the relaxation complexity. The difference is that in contrast to the relaxation complexity, where $Y = \Z^d \setminus X$ is an infinite set, in the setting of the separation problem both $X$ and~$Y$ are finite.
This makes the problem more accessible from the algorithmic perspective.

Various special versions of this problem were considered in the literature. For example, if $X \subseteq \{0,1\}^d$ and $Y = \{0,1\}^d \setminus X$, then in game theory the minimal $k$ such that $X$ can be separated from $Y$ by $k$ linear inequalities is called the the dimension of a game (cf.~\cite{taylorzwicker1999simplegames}), or threshold number of a game (cf.~\cite{hammeribarakipeled1981threshold}). 
In our notation this number equals $\rc(X,\{0,1\}^d)$, and in yet a different language, Jeroslow~\cite{jeroslow1975ondefining} proved that $\rc(X, \{0,1\}^d) \leq 2^{d-1}$ and exhibited examples that attain equality.
Hojny~\cite{hojny2018strong} studied relaxations within $\{0,1\}^d$ with respect to the size of the coefficients used.
	
The separation problem can be reduced to mixed-integer linear programming (MILP). To this end, we introduce the parameter 
\begin{equation}
	\rho := \max \{ \|x\|_\infty \,:\, x \in X\}, \label{rho:eq}
\end{equation}
where $\|x\|_\infty$ denotes the maximum norm of~$x$, so that $X \subseteq [-\rho,\rho]^d$.
We also fix the big-$M$ parameter
\begin{equation}
	M :=  2 (d \rho  + 1). \label{M:eq}
\end{equation}
We formulate a MILP that uses binary variables to encode the decision whether a given inequality separates~$X$ from a given point of $Y$. The real variables of the MILP are the coefficients of the system $Ax \le b$, and the lower bound $\mu$ is the margin by which an inequality of $A x \le b$ not valid on a point $y \in Y$ is violated:

\newcommand{\SEPMILP}{\operatorname{(SEP-MILP)}}

\begin{alignat*}{3}
\SEPMILP\qquad & \text{maximize} & \mu & & & \\
 & \text{subject to} & A x  & \le b & &  \forall x \in X, \\
 & & \mathbf{1}^\intercal s_y & \ge 1  & & \forall y \in Y, \\
 & & A y + M(\mathbf{1} - s_y) & \ge b + \mathbf{1} \mu \qquad\quad & & \forall y \in Y, \\
 & & A & \in [-1,1]^{k \times d} & & \\
 & & b & \in [-d\rho,d\rho]^k & & \\
 & & \mu & \in [0,1] & & \\
 & & s_y & \in \{0,1\}^k  & & \forall y \in Y.
\end{alignat*}

\begin{proposition} \label{prop:milp}
	Let $X$ and $Y$ be non-empty finite subsets of $\Q^d$ and let $k \in \Z_{>0}$. Then the following conditions are equivalent:
	\begin{enumerate}[label=(\roman*)]
		\item \label{X:sep:Y} The set $X$ can be separated from $Y$ by a system of $k$ linear inequalities.
		\item \label{X:sep:Y:rat} The set $X$ can be separated from $Y$ by a system of $k$ rational linear inequalities. 
		\item \label{sep:milp:cond} The mixed-integer linear problem $\SEPMILP$ with $k(d+1)+1$ real and $\card{Y} k$ binary variables and parameters $\rho$ and $M$, given by \eqref{rho:eq} and \eqref{M:eq}, respectively, has a strictly positive optimal value. 
	\end{enumerate}
	Furthermore, the following statements hold:
	\begin{enumerate}[label=(\alph*)]
		\item \label{SEPMILP->SYSTEM} If $A, b, \mu, s_y\, (y \in Y)$ is a feasible solution of $\SEPMILP$ with a strictly positive optimal value $\mu>0$, then $A x \le b$ is a system of $k$ linear inequalities that separates $X$ from $Y$. 
		\item \label{kY:polytime} If $k$ and $\card{Y}$ are fixed, then $\SEPMILP$, and by this also the separation problem with input $X$, $Y$ and $k$, can be solved in polynomial time. 
	\end{enumerate}
\end{proposition}

\begin{proof}
\ref{X:sep:Y} $\Rightarrow$ \ref{sep:milp:cond}: If $X$ can be separated from $Y$ by a system $Ax \le b$ of $k$ linear inequalities, then we can rescale each inequality of the system by an appropriate non-negative value to ensure $A \in [-1,1]^{k \times d}$, which means that each coefficient of the left-hand side lies in the range $[-1,1]$.
Afterwards, we can change $b$ to ensure that each inequality of $Ax \le b$ is attained with equality on some point of $X$, by appropriately decreasing the respective right-hand side.
After these modifications, we have $b \in [-d \rho, d \rho]^d$ and $Ax \in [-d \rho,d \rho]^d$ for every $x \in X$.

We now show that $A$ and $b$ constructed above can be extended to a feasible solution of $\operatorname{(SEP-MILP)}$ that has a positive objective value $\mu>0$.
For each $y \in Y$, there exists an index $i(y) \in [k]$ such that the $i$-th inequality $a_i^\intercal x \le b_i$ of the system $Ax \le b$ is violated on $y$, for $i=i(y)$.
We fix $s_y = e_{i(y)}$ and $\mu = \min\left\{ \{1 \} \cup \{ (b_{i(y)} - a_{i(y)}^\intercal y \, :\, y \in Y \} \right\} > 0$.
It is not hard to check that the above choice of $A, b, \mu, s_y \ (y \in Y)$ is feasible.
In fact, $A x \le b$ holds for every $x \in X$ by construction, $\mathbf{1}^\intercal s_y  \ge 1$ holds since $s_y \in \{e_1,\ldots,e_k\}$.
Let's check that $A y + M(\mathbf{1} - s_y) \ge b + \mathbf{1} \mu$ holds, too.
Consider the $i$-th inequality $a_i^\intercal y +  M(1-s_{y,i}) \ge b_i + \mu$.
If $s_{y,i}=0$, then the left hand side is at least $\rho d + 2$ in view of $a_i^\intercal y \ge - \rho d$ and $M = 2 (\rho d + 1)$, while the right hand side is at most $\rho d +1$ in view of $b_i \le \rho d$ and $\mu \le 1$.
If $s_{y,i}=1$, then the $i$-th inequality of $A x \le b$ is not valid by the margin, which is at least $\mu$, and so we see that $a_i^\intercal y + M(1-s_{y,i}) = a_i^\intercal p \ge b_i + \mu$, as desired. 
	
	\ref{sep:milp:cond} $\Rightarrow$ \ref{X:sep:Y:rat}: It is clear that the set of feasible solutions of $\SEPMILP$ is a union of finitely many rational polyhedra. This shows that in case of feasibility, the problem always has a rational optimal solution. It is straightforward to see that such a rational optimal solution yields a rational system $A x \le b$ of inequalities that separates $X$ from $Y$. 
	
	The implication \ref{X:sep:Y:rat} $\Rightarrow$ \ref{X:sep:Y} is clear. 
	
	Claim~\ref{SEPMILP->SYSTEM} follows from the interpretation of the constraints of $\SEPMILP$ that has been given in the proof above. 	As for assertion~\ref{kY:polytime}, note that if $k$ and $|Y|$ are fixed, the number of possible choices of the variables $s_y \ ( y \in Y)$ is $2^{k |Y|}$, which is a constant. Thus, by enumeration of all possible choices, solving $\SEPMILP$ gets reduced to solving $2^{k |Y|}$ linear programs. 
\end{proof}

As a consequence of the above studies, we can now prove Parts~\ref{par:complete} and~\ref{int:cond} of Theorem~\ref{thm:main-conds}.

\begin{corollary}
\label{cor:X-Y-finite-computability}
Let $X,Y \subseteq \Z^d$ be finite sets such that $X$ is lattice-convex and $X \subseteq Y$.
Then, $\rc(X,Y) = \rc_\Q(X,Y)$ and this number can be computed algorithmically.

In particular, $\rc(X) = \rc_\Q(X)$ is computable if $X$ is parity-complete or $\conv(X)$ is not lattice-free.
\end{corollary}

\begin{proof}
Follows from Propositions~\ref{prop:milp} and \ref{prop:obs-observations:XuObs}, and from the finiteness and computability of $\obs(X)$ in the case that $X$ is parity-complete or $\conv(X)$ is not lattice-free, established in Theorem~\ref{thm:parity-complete} and Theorem~\ref{thm:non-hollow-finite-obs}, respectively.
\end{proof}

Based on Theorem~\ref{thm:hollow-threshold-width} we prove Theorem~\ref{thm:main-conds}\ref{large:width} along the same lines.
At this point however, we are missing computability of $\obs(X)$ in the case of lattice-width $w(X) > w^\infty(d)$.
This will be completed with Theorem~\ref{thm:finite-obs-computability} below.

\section{Mixed-integer quantifier elimination and applications to the computation of the relaxation complexity}
\label{sect:quant-elim}
In this section, we utilize the theory of quantifier elimination towards deciding computability of the relaxation complexity of specially structured lattice-convex sets.

We first develop quantifier elimination for a special mixed-integer version of quantified boolean combinations of linear inequalities.
This will be one of the key ingredients in our proof of Theorem~\ref{thm:main-low-dims}\ref{d=3}, since it enables us to algorithmically compute the relaxation complexity for three-dimensional lattice-convex sets that have infinitely many observers, that is, those that cannot be dealt with using the tools from Section~\ref{sect:sep-milp}.
In the second part, we exploit decidability of Presburger arithmetic and devise an algorithm that computes the set of observers under the sole assumption that it is finite.

\subsection{A special quantifier elimination problem}
\label{sect:bcli}

This section is devoted to a special case of a mixed-integer linear quantifier elimination problem, that we need to settle the computability of $\rc(X)$ for three-dimensional lattice-convex sets~$X$.
The reader interested in the relaxation complexity alone may skip this section and jump right to Section~\ref{sect:d=3}.

Let $f_1,\ldots,f_m : \R^d \to \R$ be polynomial functions in~$d$ variables with coefficients in~$\Q$, and let $B : \{\true,\false\}^m \to \{\true,\false\}$ be a Boolean function in~$m$ variables.
We call the function $C : \R^d \to \{\true,\false\}$ defined by
\[
	C(x) := B \Bigl( (f_1(x) \ge 0) ,\ldots,  (f_m(x) \ge 0) \Bigr), \quad x \in \R^d,
\]
a \emph{Boolean combination of polynomial inequalities}, for short $\BCPI$.
If all the $f_i$ are affine functions, we call $C(x)$ a \emph{Boolean combination of linear inequalities}, and write $\BCLI$.
More generally, we also allow to use $f_i(x) \triangleright 0$ with $\triangleright \in \{=, \ge , \le , > , < \}$ in any combination, which does not increase the expressive power of $\BCPI$'s, because $f_i(x) > 0$ is a negation of $f_i(x) \le 0$, $f_i(x) \le 0$ is equivalent to $-f_i(x) \ge 0$, $f_i(x) = 0$ can be expressed as conjuction of $f_i(x) \ge 0$ and $-f_i(x) \ge 0$, etc.

Spurred by Hilbert's 10th Problem, it has been of great interest to decide the validity of quantified expressions of the form
\begin{equation}
\label{quant}
Q_1 \, x_1 \in R \ \cdots \ Q_k \, x_k \in R \, : \, C(x),
\end{equation}
where $Q_1,\ldots,Q_k \in \{\forall ,\exists \}$, $R \in \{\R,\Z\}$, and where the unquantified variables $x_{k+1},\ldots,x_d$ are usually called \emph{free} variables.
If in every fragment $Q_i \, x_i \in R$ in~\eqref{quant} the ring $R = \R$, we say that we consider a \emph{real} quantified expression; if $R = \Z$ in every fragment we call the expression \emph{integer}, and if both cases occur we call it a \emph{mixed-integer} quantified expression.

A very successful approach is to investigate whether a corresponding expression~\eqref{quant} admits \emph{quantifier elimination}, which means that there is an algorithm that constructs a $\BCPI$ $D(y)$ that is equivalent to~\eqref{quant}.
The presumably first result in this direction is what is nowadays called Fourier-Motzkin elimination in linear programming, establishing that every real quantified expression for $\BCLI$ admits quantifier elimination.
In particular, every such expression is equivalent to another $\BCLI$ (cf.~Schrijver~\cite[\S12.2]{schrijver1986theory}).
Much more generally, Tarski~\cite{tarski1951adecision} showed that every real quantified expression for $\BCPI$ is decidable.
By now there exist various improvements of Tarski's result, determining asymptotically fastest possible quantifier elimination algorithms for real quantified expressions for $\BCPI$ (cf.~Basu, Pollack \& Roy~\cite{basupollackroy2006algorithms}).
Regarding the case of pure integer quantifications, a landmark result is the decidability of Presburger arithmetic, meaning that every integer quantified expression for $\BCLI$ admits quantifier elimination (cf.~Presburger's original work~\cite{presburger1929vollstaendigkeit} and the excellent survey article by Haase~\cite{haase2018survival}).
On the negative side, Jeroslow~\cite{jeroslow1973therecannot} showed that Quadratic Integer Programming is undecidable.

In view of these results one may ask whether every mixed-integer quantified expression for $\BCLI$ is decidable, or if such formulas even admit quantifier elimination.
Liberti~\cite{liberti2019undecidability} discusses undecidability of general mixed-integer \emph{nonlinear} programming in great detail.
However, to the best of our knowledge the above question is not settled.

\begin{problem*}
\label{prob:mixed-integer-QE}
Let $C(x)$ be a $\BCLI$ and consider the mixed-integer quantified expression
\[
Q_1 \, x_1 \in R_1 \ \cdots \ Q_k \, x_k \in R_k \, : \, C(x),
\]
where $Q_1,\ldots,Q_k \in \{\forall ,\exists \}$ and $R_1,\ldots,R_k \in \{\R,\Z\}$.
Does there exist an algorithm that constructs a $\BCPI$ (or even a $\BCLI$) $D(y)$ equivalent to it?
\end{problem*}

%
%

Motivated by an application to determining $\rc(X,X \cup Y)$ for specially structured infinite sets~$Y$, we solve the most basic instance of this problem, where only one inner variable is allowed to be quantified over the integers.
Our proof shows that we indeed achieve quantifier elimination in this case.

\begin{theorem}
\label{BCLI:one:int}
Let $C(y,z)$ be a $\BCLI$ in $k+1$ variables $(y,z) \in \R^k \times \R$. There is an algorithm that decides the validity of the quantified statement
\begin{equation}
\label{eqn:quant-one-inner}
	\exists \, y \in \R^k \ \forall z \in \Z \,:\, C(y,z). 
\end{equation}
\end{theorem}
\begin{proof}
The main idea is to reformulate the statement as
\begin{equation}
\label{eqn:quant-one-inner-equiv}
	\exists \, (y,u) \in \R^k \times \R^m \ \exists \, v \in \Z^n \, : \, D(u,y,v),
\end{equation}
for some $\BCLI$ $D(u,y,v)$.
Once this is achieved, we can reorder the existential quantifiers
\[
	\exists \, v  \in \Z^n \ \exists \, (y,u) \in \R^k \times \R^m \, :\, D(u,y,v),
\]
then eliminate the quantifiers over real variables by Fourier-Motzkin elimination, and obtain a formula 
\[
	\exists \, v \in \Z^n \, :\, E(v),
\]
for some $\BCLI$ $E(v)$.
In this latter formula we then bring $E(v)$ into a disjunctive normal form and convert all inequalities into the form $\ge $.
Decidability of integer linear programming (cf.~Borosh \& Treybig~\cite{boroshtreybig1976bounds} or~Schrijver~\cite[Ch.~17 \& 18]{schrijver1986theory}) finally shows decidability of the equivalent original quantified statement~\eqref{eqn:quant-one-inner}.
	
For constructing a $\BCLI$ $D(u,y,v)$ such that~\eqref{eqn:quant-one-inner} is equivalent to~\eqref{eqn:quant-one-inner-equiv} we proceed as follows:
We first write $C(y,z)$ as the disjunction $C(y,z) = \bigvee_{j=1}^s C_j(y,z)$, where each $C_j(y,z)$ is a conjunction of inequalities.
Each inequality involved in $C(y,z)$ is either an inequality that depends only on~$y$, or it involves $z$ and can be written as either a lower or an upper bound on~$z$: $z \ge l(y)$ or $z > l(y)$ or $z \le u(y)$ or $z \le u(y)$. 
We want to make the setting uniform by getting rid of the strict inequalities involving~$z$. If we have a strict inequality $z  > l(y)$ we can rewrite it as a non-strict one $z \ge u_i + l(y)$ using an additional real variable $u_i \in \R$ satisfying the positivity constraint $u_i > 0$.

By this modification, we replace $C(y,z)$ by an equivalent $\BCLI$ $D(u,y,z)$, given as the disjunction $\bigvee_{j=1}^s D_j(u,y,z)$, where each $D_j(u,y,z)$ is a conjuction of inequalities, in which each inequality either involves only $(y,u)$, or is an upper bound inequality $z \le U(u,y)$, or a lower bound inequality $z \ge L(u,y)$ on~$z$.
When we fix $(y,u) \in \R^k \times \R^m$ the geometric situation is as follows: Each $D_j(u,y,z)$ defines the set $S_j(u,y):=\setcond{ t \in \R }{D_j(u,y,t)}$, which is a closed interval (possibly equal to the whole $\R$ or empty, in degenerate situations).
Our formula~\eqref{eqn:quant-one-inner} is now equivalent to
\begin{equation}
\label{eqn:quant-one-inner-final}
	\exists \, (y,u) \in \R^k \times \R^m \, : \, \left( \bigcup_{j=1}^s S_j(u,y) \supseteq \Z \right).
\end{equation}
It remains to phrase the condition $\bigcup_{j=1}^s S_j(u,y) \supseteq \Z$ as a purely existential integer quantified statement.
How do we phrase that a family of intervals covers the integers?
If some of $D_j(u,y,z)$ does not depend on $z$, then the validity of $D_j(u,y,z)$ means that $S_j(u,y) = \R$ so that $S_j(u,y) \supseteq \Z$.
If there is no such $D_j(u,y,z)$, then there must be an interval infinite to the left, and an interval infinite to the right, which together cover all but finitely many points in~$\Z$.
The remaining points are covered by the remaining intervals.

Assume $S_{j_1}(u,y)$ is the interval infinite to the left.
We can pick the maximal integer value $v_1 \in \Z$ in that interval, the successor $v_1 + 1$ will be covered by some other interval $S_{j_2}(u,y)$.
If the interval $S_{j_2}(u,y)$ is finite, there is a maximal integer value $v_2 \in \Z$ in that interval, whose successor $v_2 + 1$ will be covered by some third interval.
Repeating this process, one eventually reaches the last finite interval, with maximal integer value~$v_n$, say.
Its successor $v_n+1$ will be covered by some interval $S_{j_{n+1}}(u,y)$, which is infinite to the right.
All this is summarized as the formula $D(u,y,v)$ defined by
\[
D_{j_1}(u,y,v_1) \wedge \left( \bigwedge_{i=2}^n D_{j_i}(u,y,v_{i-1} +1) \wedge D_{j_i}(u,y,v_i) \right) \wedge D_{j_{n+1}}(u,y,v_n+1),
\]
which is valid for some $n < s$ and some $j_1,\ldots,j_{n+1} \in \{1,\ldots,s\}$.
Moreover, $D_{j_1}(u,y,v_1)$ contains only upper bounds on $z$, $D_{j_{n+1}}(u,y,v_n+1)$ contains only lower bounds on $z$, while $D_{j_i}(u,y,v_{i-1}+1)$ and $D_{j_i}(u,y,v_i)$, with $1 < i \le n$, contain both lower and upper bounds on $z$.

We thus proved that $\bigcup_{j=1}^s S_j(u,y) \supseteq \Z$ is equivalent to the expression $\exists \, v \in \Z^n : D(u,y,v)$, which finishes the proof.
\end{proof}

We now obtain our desired application to determining $\rc(X,X \cup Y)$ for specially structured infinite sets~$Y$.
	
\begin{corollary}
\label{cor:rc-special-Y}
	Let $k \in \N$, let $X \subseteq \Z^d$ be a finite lattice-convex set, and let $Y \subseteq \Z^d \setminus X$ be of the form $Y = Y_0 \cup (L_1 \cap \Z^d) \cup \ldots \cup (L_m \cap \Z^d)$, where~$Y_0$ is finite and $L_1,\ldots,L_m$ are lines, each containing a lattice point~$p_i$, and all being parallel to some common primitive vector $u \in \Z^d \setminus \{0\}$.
	Then, there is an algorithm that decides whether there is a system of $k$ linear inequalities $f_1(x) \ge 0,\ldots, f_k(x) \ge 0$, that is satisfied for every $x \in X$ and is not satisfied for every $y \in Y$.
	
	In particular, for given $X$ and $Y$ as above, there is an algorithm that determines $\rc(X,X \cup Y)$. 
\end{corollary}

\begin{proof}
The lattice points of $L_i$ are parametrized as $p_i + z u$ with $z \in \Z$.
A priori, the expression $f_j(p_i + z u)$ that will show up in our considerations, is non-linear when we view $f_j$ as a variable vector (in the vector space of affine functions on $\R^d$) and $z$ as an integer variable.
When we write $f_j(x) = a_j^\intercal x + b_j$, for some $a_j \in \R^d$ and $b_j \in \R$, we have
\[
f_j(p_i + z u) = a_j^\intercal (p_i + z u) + b_j = (a_j^\intercal u) z + a_j^\intercal p_i + b_j.
\]
At this point, by rescaling~$f_j$, that is, rescaling the vector $(a_j^\intercal, b_j) \in \R^{d+1}$, we can always assume that $a_j^\intercal u = \epsilon_j$, where $\epsilon_j \in \{-1,0,1\}$, is a linear equality in $a_j$.
The three choices for each $\epsilon_j$ produce $3^k$ choices for $\epsilon=(\epsilon_1,\ldots,\epsilon_k)$.
For a given choice, the expression $f_j(p_i + zu) = \epsilon_j z + a_j^\intercal p_i + b_j$ is linear in the variables $z, a_j, b_j$.
Therefore, for a fixed $\epsilon \in \{-1,0,1\}^k$, we can model the statement that there is a system of $k$ linear inequalities $f_1(x) \ge 0,\ldots, f_k(x) \ge 0$, that is satisfied for every $x \in X$ and is not satisfied for every $y \in Y$, by the following mixed-integer quantified expression:
\begin{align}
\exists \, a \in \R^{d \times k} \ \exists \, b \in \R^k \ \forall z \in \Z \,:\, C_\epsilon(a,b,z),\label{eqn:rc-expression}
\end{align}
where $a = (a_1,\ldots,a_k)$ and $b = (b_1,\ldots,b_k)^\intercal$.
The $\BCLI$ $C_\epsilon(a,b,z)$ appearing in this expression is defined by
\[
C_\epsilon(a,b,z) = C_X(a,b) \wedge C_0(a,b) \wedge \bigwedge_{i=1}^m C_i^\epsilon(a,b,z),
\]
with the following constituents:
\[
C_X(a,b) = \bigwedge_{x \in X} \left( (f_1(x) \geq 0) \wedge \ldots \wedge (f_k(x) \geq 0) \right)
\]
and
\[
C_0(a,b) = \bigwedge_{y \in Y_0} \neg \left( (f_1(y) \geq 0) \wedge \ldots \wedge (f_k(y) \geq 0) \right),
\]
model that the linear system is satisfied by all $x \in X$, and by none of the $y \in Y_0$, respectively.
And, for each $i \in [m]$, the $\BCLI$
\begin{align*}
C_i^\epsilon(a,b,z) &= \neg \left( (f_1(p_i + z u) \geq 0) \wedge \ldots \wedge (f_k(p_i + z u) \geq 0) \right) \\
&= \neg \left( (\epsilon_1 z + a_1^\intercal p_i + b_1 \geq 0) \wedge \ldots \wedge (\epsilon_k z + a_k^\intercal p_i + b_k \geq 0) \right),
\end{align*}
models that the linear system is not satisfied for the lattice point $p_i + z u$ on the line~$L_i$.

Since the expression~\eqref{eqn:rc-expression} has the right form to apply Theorem~\ref{BCLI:one:int}, we obtain decidability of the representation of $X$ within $X \cup Y$ with $k$ linear inequalities, by considering $3^k$ such expressions, one for each $\epsilon \in \{-1,0,1\}^k$.
By a standard binary search, we can thus compute $\rc(X,X \cup Y)$.
\end{proof}

\subsection{Computing finite sets of observers}

A set $\cS \subseteq \Z^d$ of lattice points is \emph{Presburger definable} if there exists a $\BCLI$ $C(x)$ in~$d$ variables $x = (x_1,\ldots,x_d)$ such that $\cS = \setcond{x \in \Z^d}{C(x)}$.
Given a lattice point $a \in \Z^d$ and some vectors $v_1,\ldots,v_\ell \in \Z^d$, the set $a + \setcond{\lambda_1 v_1 + \ldots + \lambda_\ell v_\ell}{\lambda_1,\ldots,\lambda_\ell \in \Z_{\geq0}}$ is called an \emph{elementary set}.
Such sets are easily seen to be Presburger definable and an elegant result of Ginsburg \& Spanier~\cite{ginsburgspanier1964bounded} (cf.~\cite[Sect.~4]{haase2018survival}) shows that Presburger definable subsets of~$\Z^d$ are precisely the finite unions of elementary sets.
Note also that the complement of a Presburger definable set $\cS = \setcond{x \in \Z^d}{C(x)}$ is given by $\Z^d \setminus \cS = \setcond{x \in \Z^d}{\neg\,  C(x)}$, and is therefore itself Presburger definable.

%

With these preparations we can now formulate and prove a result that was already hinted at in the end of Section~\ref{sect:observers}.
This result will also complete the proof of Theorem~\ref{thm:main-conds}\ref{large:width}.

\begin{theorem}
\label{thm:finite-obs-computability}
If a finite lattice-convex set $X \subseteq \Z^d$ is such that $\obs(X)$ is finite, then there is an algorithm that computes $\obs(X)$ and $\rc(X) = \rc_\Q(X)$.
\end{theorem}

\IncMargin{1em}
\begin{algorithm}[ht]
 \DontPrintSemicolon
 \SetKwInOut{Input}{Input}\SetKwInOut{Output}{Output\,}
 \SetKwData{GoTo}{go to}

 \Input{A lattice-convex set $X \subseteq \Z^d$ for which $\obs(X)$ is finite.}
 \Output{$\obs(X)$}
 \BlankLine
 $\cS := X$ and $\cO := \emptyset$\;\nllabel{line:init}
 \While{$\exists\, q \in \Z^d \setminus \cS$}{\nllabel{line:while}
  $Q_q := \conv(X \cup \{q\})$\;\nllabel{line:Qq}
  \eIf{$\exists\, p \in Q_q \cap \Z^d \setminus (X \cup \{q\})$}{\nllabel{line:if}
   $q := p$\;
   \GoTo~\ref{line:Qq}\;\nllabel{line:goto}
   }{
   $\cO := \cO \cup \{q\}$ \tcp*[r]{$q$ is an observer of $X$}
   $C_q := q + \setcond{\sum_{x \in X} \lambda_x(q-x)}{\lambda_x \geq 0, x \in X}$\;\nllabel{line:cone}
   compute a facet description $C_q = \setcond{y \in \R^d}{a_i^\intercal y \leq b_i, 1 \leq i \leq m}$, where $a_i \in \Z^d$ and $b_i \in \Z$, for $1 \leq i \leq m$\;
   $\cS := \cS \cup (C_q \cap \Z^d)$\;\nllabel{line:augmentS}
  }
 }\nllabel{line:endwhile}
 \Return{$\cO$}\nllabel{line:return}
 \caption{Computing the set $\obs(X)$ of observers.}
 \label{algo:observers}
\end{algorithm}
\DecMargin{1em}

\begin{proof}
It suffices to show that we can compute the set $\obs(X)$, because if this is done then Corollary~\ref{cor:X-Y-finite-computability} implies that $\rc(X) = \rc_\Q(X)$ and that this number is computable. 
A pseudocode of our algorithm to compute $\obs(X)$ is given in Algorithm~\ref{algo:observers}.

We first argue that Algorithm~\ref{algo:observers} is correct:
Observe that in every stage of the algorithm the set~$\cS$ consists of those lattice points for which we have currently decided whether they belong to the set of observers $\obs(X)$.
At the initialization step in Line~\ref{line:init} this is clear, because none of the points in~$X$ is an observer.
If a point $q \in \Z^d \setminus \cS$ is found, then $Q_q$ may or may not contain lattice points besides $X \cup \{q\}$.
The loop through Lines~\ref{line:Qq}--\ref{line:goto} successively reduces $Q_q$ until $Q_q \cap \Z^d = X \cup \{q\}$, which means that $q \in \obs(X)$.
By definition of the cone $C_q$ in Line~\ref{line:cone}, $q$ is the only observer it contains.
Thus, we can safely disregard the other lattice points in $C_q$ and thus augment $\cS$ accordingly in Line~\ref{line:augmentS}.
This establishes the claimed invariance property of~$\cS$.
These arguments also show that each new iteration in the while loop through Lines~\ref{line:while}--\ref{line:endwhile} provides us with a new observer.
Since by assumption $\obs(X)$ is finite, this means that Algorithm~\ref{algo:observers} terminates after finitely many steps, and that indeed the returned set~$\cO$ in Line~\ref{line:return} equals~$\obs(X)$.

Finally, we need to make sure that the conditions in Lines~\ref{line:while} and~\ref{line:if} of Algorithm~\ref{algo:observers} are decidable, and that all occuring sets and points are computable.
Decidability of the two conditions hold since they are quantified expressions in Presburger arithmetic.
Indeed, at each stage of the algorithm, the set $\cS$ is a finite union of elementary sets: In Line~\ref{line:init}, $\cS = X$ which is a finite set and thus a finite union of elementary sets; and the augmentation in Line~\ref{line:augmentS} just expands~$\cS$ by the elementary set $C_q \cap \Z^d$.
Note also that the decision procedures for Presburger arithmetic provide us with solutions~$q$ and~$p$ in Lines~\ref{line:while} and~\ref{line:if}, respectively.
Thus, all the points and sets in Algorithm~\ref{algo:observers} are explicitly computable.
\end{proof}

\section{Computability of the relaxation complexity for \texorpdfstring{$d=3$}{d=3}}
\label{sect:d=3}
We are now well-prepared to demonstrate computability of the relaxation complexity of three-dimensional lattice-convex sets.
We rephrase Theorem~\ref{thm:main-low-dims}\ref{d=3} for the reader's convenience.

\begin{theorem}
\label{thm:computability-d-3}
For every full-dimensional finite lattice-convex set $X \subseteq \Z^3$, there is a finite algorithm that computes $\rc(X)$.
\end{theorem}

Given a $d$-dimensional lattice-free lattice polytope $P$, there exists a lattice-free lattice polyhedron~$\cM$ with $P \subseteq \cM$ and such that $\cM$ is inclusion-maximal among all lattice-free lattice polyhedra (see~\cite{averkovwagnerweismantel2011maximal} and~\cite{nillziegler2011projecting}).
Moreover, the same sources show that there are only finitely many choices for $\cM$, up to unimodular equivalence.

In dimension three the possible choices for $\cM$ have been classified in the series of papers~\cite{averkovwagnerweismantel2011maximal} and~\cite{averkovkruempelmannweltge2017notions}:
Up to unimodular equivalence, these choices are the \emph{slab} $[0,1] \times \R^2$, the \emph{toblerone} $\cM_0 = \conv(\{0,2 e_1, 2 e_2\}) \times \R$, and $12$ bounded lattice polytopes $\cM_1,\ldots,\cM_{12}$ of volume at most $6$ and lattice-width at least~$2$.

Our proof of Theorem~\ref{thm:computability-d-3} is based on this classification and on a case distinction on the structural properties of the observers of the given lattice-convex set~$X$.
To this end, given an observer $p \in \obs(X)$, we write
\[
Q_p = \conv(X \cup \{p\}).
\]
Note that~$Q_p$ is a full-dimensional lattice polytope and that~$p$ is one of its vertices.
Further, we write $\D_X \subseteq \Z^d \setminus \{0\}$ for the set of \emph{lattice-width directions} of $X \subseteq \Z^d$, that is, $ u \in \Z^d \setminus \{0\}$ is contained in~$\D_X$ if and only if $w(X)=w(X,u)$.
The set $\D_X$ is the explicitly computable set 
\[
	\D_X = ( w(X) \cdot (X-X)^\star ) \cap \Z^d \setminus \{0\},
\]
where  $Y^\star := \setcond{x \in \R^d}{x^\intercal y \leq 1 \text{ for all } y \in Y}$ denotes the \emph{polar} of a set $Y \subseteq \R^d$.
Also, $\D_X$ has at most $3^d-1$ elements as shown in~\cite{draismamcallisternill2012lattice}.
Each such $u \in \D_X$ corresponds to a pair of parallel supporting lattice planes $H_u^+$ and~$H_u^-$ of~$\conv(X)$ with outer normals $u$ and $-u$, respectively.
We say that $u \in \D_X$ has \emph{type $(i,j)$} if $\dim(X \cap H_u^+) = i$ and $\dim(X \cap H_u^-) = j$, or vice versa.

With this notation, we can now describe the structure of the proof of Theorem~\ref{thm:computability-d-3}.
We distinguish different types of observers $p \in \obs(X)$.
In most of the cases we describe a finite and computable search space for the respective observer~$p$, which then allows us to use Corollary~\ref{cor:X-Y-finite-computability} and Proposition~\ref{prop:obs-observations:XuObs} to compute $\rc(X)$.
Only in the Cases 2.2.1, 2.2.2, and 2.2.3 we need to employ a different method to compute $\rc(X)$ either directly or by an alternative algorithm.
The cases are as follows:

\begin{enumerate}[label = \underline{Case~\arabic*:}, itemsep=5pt]

 \item $Q_p$ is not lattice-free

 \item $Q_p$ is lattice-free

 \begin{enumerate}[label = \underline{Case~2.\arabic*:}, leftmargin=20pt, itemsep=5pt, topsep=5pt]

  \item $w(Q_p) > 1$ 

   \begin{enumerate}[label = \underline{Case~2.1.\arabic*:}, leftmargin=20pt, itemsep=5pt, topsep=5pt]

    \item up to unimodular equivalence, $Q_p$ is contained in $\cM_0$ 

    \item up to unimodular equivalence, $Q_p$ is contained in one of the polytopes $\cM_1,\ldots,\cM_{12}$

   \end{enumerate}

  \item $w(Q_p) = 1$

   \begin{enumerate}[label = \underline{Case~2.2.\arabic*:}, leftmargin=20pt, itemsep=5pt, topsep=5pt]

    \item there exists a lattice-width direction of $X$ of type $(1,1)$

    \item there exists a lattice-width direction of $X$ of type $(2,0)$

    \item there exists a lattice-width direction of $X$ of type $(2,1)$

    \item every lattice-width direction of $X$ is of type $(2,2)$

   \end{enumerate}

 \end{enumerate}

\end{enumerate}

\noindent Let us now get into the details.

\subsection*{Details for Case 1}

Theorem~\ref{thm:non-hollow-finite-obs} and its proof show that if $Q_p$ is not lattice-free, then $p \in X + c_d \cdot \conv(X-X)$, with the constant~$c_d = d(2d+1)(s_{2d+1}-1)$.
Hence, the set
\[
\cS_1 = \left( X + c_3 \cdot \conv(X-X) \right) \cap \Z^3,
\]
with the exact value $c_3 = 21 \cdot (s_7-1) = 2.23651195966947 \cdot 10^{14}$, may be taken as a finite search space in this case, which of course is explicitly computable.
Clearly, the constant $c_3$ chosen above is tremendously large.
This is due to our limited knowledge on the best upper bound on the coefficient of asymmetry in~\eqref{eqn:ca-bound}.
If the conjectured bound in Remark~\ref{rem:pikhurkos-guess} were true, we could replace $c_3$ with the constant $s_3^2-2 = 47$.

\subsection*{Details for Case 2.1.1}

By assumption, we have $w(Q_p) \geq 2$ and a unimodular copy of $Q_p$ is contained in the toblerone $\cM_0$, which has lattice-width~$2$.
Therefore, $w(Q_p)=2$.
In order to find a finite search space that is guaranteed to contain the observer~$p$, we need to identify \emph{all} unimodular copies of~$Q_p$ that lie in~$\cM_0$, or equivalently, all unimodular copies of~$\cM_0$ that contain~$Q_p$.

\begin{lemma}
\label{lem:toblerones}
For a subset $Y \subseteq \Z^3$ of lattice-width $w(Y) = 2$, let
\[
T(Y) := \setcond{\varphi(\cM_0)}{Y \subseteq \varphi(\cM_0) \text{ and } \varphi : \R^3 \to \R^3 \text{ is a unimodular map}}
\]
be the family of unimodular copies of the toblerone~$\cM_0$ that contain~$Y$, and assume that $T(Y) \neq \emptyset$.
Then, $T(Y)$ is finite and explicitly computable.
\end{lemma}

\begin{proof}
First of all, the toblerones $T \in T(Y)$ are in correspondence with pairs $(u,v)$ of lattice-width directions of~$Y$ such that
\begin{enumerate}[label=(\roman*)]
 \item $u+v$ is a lattice-width direction of~$Y$ as well, and
 \item $\{u,v\}$ is a basis of the two-dimensional lattice $\lin(\{u,v\}) \cap \Z^3$.
\end{enumerate}
The lattice-width directions of~$Y$ are the non-zero vectors in the set $\Z^3 \cap 2 (Y-Y)^\star$ which is finite and computable.
This already shows the claimed finiteness of $T(Y)$.

Condition~(i) on a pair $(u,v)$ can be checked by a membership test of $u+v$ in~$2 (Y-Y)^\star$.
Condition~(ii) can be computationally checked by computing the row-style Hermite normal form of the $(3 \times 2)$-matrix with columns $u,v$.
The set $\{u,v\}$ is a basis of $\lin(\{u,v\}) \cap \Z^3$ if and only if this Hermite normal form equals the matrix with columns $e_1,e_2$ (cf.~\cite[Ch.~4]{schrijver1986theory} for more details).
\end{proof}

Let us now fix one toblerone $T \in T(Q_p)$.
With the notation of the proof of Lemma~\ref{lem:toblerones}, we write $w=u+v$ for a pair $(u,v)$ of lattice-width directions of~$Q_p$ corresponding to~$T$.
Moreover, denote the three unbounded facets of~$T$ that are orthogonal to $u,v$ and $w$ by $F_u,F_v$ and $F_w$, respectively.
Let $\pi$ be the orthogonal projection onto the hyperplane $\lin(\{u,v\})$.
There are two options: Either $\phi(\pi(X)) = \conv(\{0,2e_1,2e_2\})$ or exactly one of the vertices $0$, $2e_1$, $2e_2$ is missing in the projection $\phi(\pi(X))$.
Indeed, if two such vertices were missing, then $w(Q_p)=1$, contradicting our assumption.

Write $Q = \conv(X)$ as usual.
In the first case, all the intersections $Q \cap F_u$, $Q \cap F_v$, and $Q \cap F_w$ are lattice polygons of dimension $1$ or $2$, and the observer~$p$ may a priori lie anywhere on the facets $F_u$, $F_v$, and $F_w$.
Without loss of generality, assume that $p \in F_u$, and let $E$ be an edge of $Q \cap F_u$ with the property that the intersection of $\conv(E \cup \{p\})$ and $Q \cap F_u$ is one-dimensional.
Since $p$ is an observer of $X$, it can only lie in the next parallel lattice line to $\aff(E)$, as otherwise the lattice triangle $\conv(E \cup \{p\})$ would have too large an area and contain additional lattice points.
This follows, for instance, from Pick's Theorem that relates the area of a lattice polygon to the number of lattice points that it contains (cf.~\cite[Sect.~19]{gruber2007convex}).

In the second case, again all the intersections $Q \cap F_u$, $Q \cap F_v$, and $Q \cap F_w$ are lattice polygons of dimension $1$ or $2$, but now the observer $p$ is constrained to lie in an unbounded edge, say $F_u \cap F_v$, of the toblerone ~$T$.
Again, let $E$ be an edge of $Q \cap F_u$ or $Q \cap F_v$ with the property described before.
Similarly as in the first case, $p$ must lie in the lattice line $\aff(E)$ or in the next parallel lattice line to it, in order not to contradict the property of being an observer.

Summarizing our considerations, in both cases there is only a finite and explicitly computable set $\cS_T$ of lattice points that are candidates for the observer~$p$.
Hence a finite and computable search space that is guaranteed to contain the observer~$p$ is
\[
\cS_{2.1.1} = \bigcup_{T \in T(Q_p)} \cS_T.
\]

\subsection*{Details for Case 2.1.2}

Since by assumption a unimodular copy of~$Q_p$ is contained in one of the lattice polytopes $\cM_1,\ldots,\cM_{12}$ of volume at most $6$, the observer~$p \in \obs(X)$ cannot lie too far away from~$X$ in the following sense:
Write $Q = \conv(X) = \setcond{x \in \R^3}{a_i^\intercal x \leq b_i, 1 \leq i \leq m}$, where the $a_i \in \Z^3$ are primitive outer normal vectors of the facets of~$Q$, and where $b_i \in \Z$ for $1 \leq i \leq m$.
Since every facet~$F$ of~$Q$ is a lattice polygon and the pyramid $F_p = \conv(F \cup \{p\})$ is a lattice $3$-polytope of volume at most~$6$, the point~$p$ can lie at most $36$ parallel lattice hyperplanes away from $\aff(F)$ (note that a lattice $3$-simplex has volume at least $1/6$).
Therefore, the finite search space for~$p$ in this case can be taken as
\[
\cS_{2.1.2} = \setcond{x \in \R^3}{a_i^\intercal x \leq b_i+36, 1 \leq i \leq m} \cap \Z^3.
\]

\subsection*{Details for Case 2.2.1}

If there is a lattice-width direction of~$X$ of type $(1,1)$, then $\conv(X)$ is a tetrahedron and as such a rational relaxation of~$X$ with four facets.
Since $X$ is also full-dimensional, Corollary~\ref{cor:rc>=d+1:d<=4} yields that $\rc_\Q(X) = \rc(X) = 4$.

\subsection*{Details for Case 2.2.2}

For sets $X$ with a lattice-width direction of type $(2,0)$ we need an auxiliary lemma that is valid in every dimension.

\begin{lemma}
\label{lem:pyramids}
Let $X' \subseteq \Z^{d-1}$ be a finite lattice-convex set and let $x^\star \in \Z^d$ be a point whose last coordinate equals~$1$.
If $\rc(X') \geq d+1$, then the lattice-convex set $X = (X' \times \{0\}) \cup \{x^\star\} \subseteq \Z^d$ satisfies $\rc(X) = \rc(X')$.
\end{lemma}

\begin{proof}
Let $P' \subseteq \R^{d-1}$ be a polyhedron with $m = \rc(X') \geq d+1$ facets such that $P' \cap \Z^{d-1} = X'$, and let $a'_1,\ldots,a'_m$ be a set of outer normal vectors of the facets of~$P'$.
After a perturbation of the facets of $P'$, we may assume that there are $a'_1,\ldots,a'_d$, say, such that the origin $0$ is contained in the relative interior of $\conv(\{a_1',\ldots,a'_d\})$.
Writing $F'_1,\ldots,F'_d$ for the corresponding facets of $P'$, this yields that the hyperplanes $H_i = \aff(F'_i \cup \{x^\star\})$, $1 \leq i \leq d$, separate $X$ from all points in $\left\{z \in \Z^d : z_d > 0\right\} \setminus X$.
This holds since $x^\star$ has height $1$ over $P'$.

Without loss of generality, we may assume that the projection of $x^\star$ onto the first $d-1$ coordinates is contained in~$P'$.
In this setting we can separate $X$ from the lattice points with negative $d$th coordinate by taking as additional hyperplane normals $a_i = (a'_i,0)$, for every $i \geq d+2$, and $a_{d+1} = (a'_{d+1},-M)$, for some large enough $M$ such that the corresponding hyperplane is almost horizontal and thus cuts off all lattice points ``below''~$P'$.

The lattice points outside of $X$ which have vanishing last coordinate are separated from~$X$, because in the process above we just tilted all the facets of~$P'$, thus leaving the separation within the plane $\{x \in \R^d : x_d = 0\}$ unchanged.

We showed so far that $\rc(X) \leq m = \rc(X')$.
If this would be strict however, then we would find a relaxation of $X$ within $\Z^d$ with~$k < m$ hyperplanes, whose restriction to the plane $\{x \in \R^d : x_d = 0\}$ would lead to a relaxation of $X'$ within $\Z^{d-1}$ with $\leq k$ hyperplanes, a contradiction to the definition of~$m$.
\end{proof}

Now, if $X \subseteq \Z^3$ has a lattice-width direction of type $(2,0)$, then up to unimodular equivalence, there is some two-dimensional finite lattice-convex set $X' \subseteq \Z^2$ and some $x^\star \in \Z^3$ with last coordinate equal to~$1$, such that $X = (X' \times \{0\}) \cup \{x^\star\}$.
In view of Weltge's~\cite[Thm.~7.5.7]{weltge2015diss} treatment of planar lattice-convex sets, we know that $\rc(X')=\rc_\Q(X') \geq 3$ and that this number is computable.
So, if $\rc(X') \geq 4$, then Lemma~\ref{lem:pyramids} implies that $\rc(X) = \rc(X')$ is computable.
If, however, $\rc(X') = 3$ and $P' \subseteq \R^2$ is a relaxation of~$X'$ within~$\Z^2$ having three facets, then $P = \conv\left((P' \times \{0\}) \cup \{x^\star\}\right)$ is a relaxation of~$X$ within~$\Z^3$ having four facets.
Since $X$ is full-dimensional, Corollary~\ref{cor:rc>=d+1:d<=4} guarantees that every of its relaxations need to have at least four facets and thus $\rc(X)=\rc_\Q(X)=4$.

\subsection*{Details for Case 2.2.3}

Since we dealt with the cases of lattice-width directions of type $(1,1)$ and $(2,0)$ before, we assume in the sequel that $X$ does not have any of such lattice-width directions, and hence every $u \in \D_X$ is of type either $(2,1)$ or $(2,2)$.
By assumption we have $w(Q_p)=1$, and thus also $w(X)=1$.
In general for sets~$X \subseteq \Z^3$ of lattice-width $w(X)=1$, the observers are contained in finitely many affine subspaces.
Recall that for each $u \in \D_X$, we denote by~$H_u^+$ and~$H_u^-$ the parallel supporting lattice planes of~$\conv(X)$.

\begin{lemma}
\label{lem:structure-of-observers}
Let $X \subseteq \Z^3$ be a full-dimensional lattice-convex set of lattice-width $w(X)=1$.
Then, there is a finite subset $Y_0 \subseteq \Z^3$ such that
\[
\obs(X) \subseteq Y_0 \cup \bigcup_{u \in \D_X} (H_u^+ \cap \Z^3) \cup (H_u^- \cap \Z^3).
\]
More precisely, if every lattice-width direction of~$X$ is of type either $(2,1)$ or $(2,2)$, then there are lattice lines $L_1,\ldots,L_m$ such that
 \[
 \obs(X) \subseteq Y_0 \cup (L_1 \cap \Z^3) \cup \ldots \cup (L_m \cap \Z^3),
 \]
 and the set~$Y_0$ and a lattice point on and the direction of every line~$L_i$ can be computed explicitly.
\end{lemma}

\begin{proof}
For $p \in \obs(X)$, let $Q_p = \conv(X \cup \{p\})$ be as above.
If $w(Q_p) > 1$, then by Case~1 and Case~2.1 there are only finitely many choices for the observer~$p$, and we may take the finite set in the claim as
\[
Y_0 = \cS_1 \cup \cS_{2.1.1} \cup \cS_{2.1.2}.
\]
If $p \in \obs(X)$ is such that $w(Q_p) = 1$, then $p \in H_u^+ \cup H_u^-$ for some $u \in \D_X$, and the claimed inclusion follows.

Now assume that every lattice-width direction of~$X$ is of type either $(2,1)$ or $(2,2)$.
Whenever $\dim(H_u^\pm \cap X) = 2$, then by Proposition~\ref{prop:finite-observers-dim-2} the corresponding plane contains only finitely many observers of~$X$, which are computable and which we may include into the finite set~$Y_0$.
If $\dim(H_u^+ \cap X) = 1$, say, then every observer of $X$ that is contained in $H_u^+$ lies on one of three consecutive parallel lattice lines.
To see this, let $L_0 = \aff(H_u^+ \cap X)$ and let~$L_+$ and $L_-$ be the two neighboring lattice lines to $L_0$ in~$H_u^+$.
Every lattice point on either~$L_+$ or $L_-$ is in fact an observer, and there are exactly two observers in $L_0$ (the lattice points next to the endpoints of $\conv(H_u^+ \cap X)$).
If $p \in H_u^+ \cap \Z^3$ is not contained in either of the three lines $L_-,L_0,L_+$, then $\conv\left((H_u^+ \cap X) \cup \{p\}\right)$ is a lattice triangle in $H_u^+$ that is not unimodularly equivalent to $\conv(\{0,e_1,e_2\})$.
This triangle must therefore have an additional lattice point, so that, in fact,~$p$ cannot be an observer of~$X$.
Computability of the lines $L_-,L_0$, and $L_+$ follows immediately from their definition.
\end{proof}

Now, if $X$ has a lattice-width direction of type $(2,1)$, then we use
Proposition~\ref{prop:obs-observations:XuObs} for the set
\[
Y = Y_0 \cup (L_1 \cap \Z^3) \cup \ldots \cup (L_m \cap \Z^3)
\]
from Lemma~\ref{lem:structure-of-observers}, and obtain that $\rc(X) = \rc(X,X \cup Y)$.
If the lattice lines $L_1,\ldots,L_m$ are parallel, we can apply Corollary~\ref{cor:rc-special-Y} and compute~$\rc(X)$ by the quantifier elimination procedure outlined in Section~\ref{sect:bcli}.
The proof of Lemma~\ref{lem:structure-of-observers} shows that this holds if all the one-dimensional sets $H_u^+ \cap X$, where $u \in \D_X$ are the lattice-width directions of~$X$ of type $(2,1)$, are parallel.

If the lattice lines $L_1,\ldots,L_m$ are not parallel, then $X$ belongs to an explicit parametrized family.
For its description, we fix some more notation:
For a lattice-width direction~$u \in \D_X$ of type $(2,1)$, write $F(X,u)$ for the set of those $x \in X$ such that $u^\intercal x$ is maximized on~$X$.
Also, we assume in the sequel that $u$ is oriented such that $\dim(F(X,u))=1$, and thus $\dim(F(X,-u))=2$.

\begin{lemma}
\label{lem:outsider}
Let $X \subseteq \Z^3$ be a finite lattice-convex set admitting linearly independent lattice-width directions $u,v \in \D_X$ of type $(2,1)$ such that $F(X,u)$ and $F(X,v)$ are not parallel.
Then, $X$ is unimodularly equivalent to
\[
X_{a,b} = \{e_1,e_2,e_1+e_2\} \cup \setcond{k e_3}{k=a,\ldots,b},
\]
for some $a,b \in \Z$ with $a \leq b$ and $\{a,b\} \neq \{0\}$.
\end{lemma}

\begin{proof}
First of all, we can apply a unimodular transformation and assume that both $u$ and $v$ are orthogonal to~$e_3$.
Let $\pi:\R^3\to\R^2$ be the projection $\pi(x_1,x_2,x_3)=(x_1,x_2)$ that forgets the last coordinate.
Since $w(X,u)=w(X,v)=1$, the projection $\pi(X)$ is unimodularly equivalent to $\{0,e_1,e_2\}$ or $\{0,e_1,e_2,e_1+e_2\}$.
The first case cannot happen, because then both $F(X,u)$ and $F(X,v)$ would be parallel to~$e_3$, contradicting our assumption.

So, we may assume that $\pi(X) = \{0,e_1,e_2,e_1+e_2\}$, and that $u=e_1$ and $v=e_2$.
Since $F(X,e_1)$ and $F(X,e_2)$ are not parallel, one of them, say $F(X,e_1)$, is not parallel to~$e_3$.
Then,
\[
\pi(F(X,-e_1)) = \{0,e_2\} \quad \text{and} \quad \pi(F(X,e_1)) = \{e_1,e_1+e_2\}.
\]
This implies that $F(X,e_2)$ is also not parallel to~$e_3$.
Indeed, if $F(X,e_2)$ would be parallel to~$e_3$, then $\pi(F(X,e_2))$ would be a single point, either equal to~$\{e_2\}$ or $\{e_1+e_2\}$.
In the first case, the set $F(X,e_2)$ would be two-dimensional, and in the second case, $F(X,e_1)$ would be two-dimensional, a contradiction either way.
We conclude that
\[
\pi(F(X,e_1)) = \{e_1,e_1+e_2\} \quad \text{and} \quad \pi(F(X,e_2)) = \{e_2,e_1+e_2\}.
\]
Furthermore, $F(X,e_1)$ and $F(X,e_2)$ share exactly one point that gets projected onto~$e_1+e_2$.
Thus, we have that $F(X,e_1) = \{p,q\}$ and $F(X,e_2) = \{q,r\}$, where $\pi(p) = e_1$, $\pi(q) = e_1+e_2$, and $\pi(r) = e_2$.

Applying a suitable unimodular transformation, we may thus assume that $p=e_1$, $q=e_1+e_2$, and $r=e_2$, and thus $X = X_{a,b}$, for some $a,b\in\Z$.
\end{proof}

In order to finish up Case 2.2.3 of the proof of Theorem~\ref{thm:computability-d-3}, we explicitly determine the relaxation complexity of the exceptional examples in Lemma~\ref{lem:outsider}.

\begin{lemma}
\label{lem:outsider-rc}
For every $a,b \in \Z$ with $a \leq b$ and $\{a,b\} \neq \{0\}$, the relaxation complexity of the set
\[
X_{a,b} = \{e_1,e_2,e_1+e_2\} \cup \setcond{k e_3}{k=a,\ldots,b}
\]
is given by $\rc(X_{a,b}) = \rc_\Q(X_{a,b}) = 4$.
\end{lemma}

\begin{proof}
The lower bound $\rc(X_{a,b}) \ge 4$ follows by Corollary~\ref{cor:rc>=d+1:d<=4} since $X_{a,b}$ is clearly full-dimensional.
For the upper bound, we construct an explicit relaxation of $X_{a,b}$ with four facets.

It is enough to consider the following cases:
\begin{enumerate}
	\item $0= a < b$,
	\item $0 < a \leq b$,
	\item $a < 0 < b$.
\end{enumerate}

\noindent \emph{Case $0=a< b$:}
The triangle $T$ with vertices 
\[
	(-1/2,-1/2), (1/4,7/4), (7/4,1/4)
\]
is a relaxation of $[0,1]^2$ with the property that the vertex $0$ of $[0,1]^2$ lies in the interior of $T$, while the other three vertices lie in the boundary of $T$.
It is clear that the tetrahedron~$P$ with base $T \times \{0\}$ and apex $b e_3$ is a relaxation of $X_{a,b} = X_{0,b}$.

\begin{figure}
\begin{center}
\begin{tikzpicture}
	\foreach \x in {1/3,2/3,3/3} {
		\draw[red,thick] ({\x*(-1/2)},{\x*(-1/2)}) -- ({\x*(1/4)},{\x*(7/4)}) -- ({\x*(7/4)},{\x*(1/4)}) -- cycle;
	}
	\fill (0,0) circle (0.08);
	\fill (1,0) circle (0.08);
	\fill (0,1) circle (0.08);
	\fill (1,1) circle (0.08);
	\draw[gray] (0,0) -- (1/4,7/4);
	\draw[gray] (0,0) -- (7/4,1/4);
	\draw[gray] (0,0) -- (-1/2,-1/2);
\end{tikzpicture} 
\end{center}
\caption{The relaxation complexity of $X_{0,3}$ is four. The figure shows the base and the two horizontal cross-sections of the tetrahedron $P$ that relaxes $X_{0,3}$.}
\end{figure}
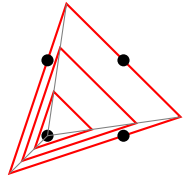

\smallskip
\noindent \emph{Case $0 < a \leq b$:}
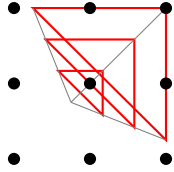
\begin{figure}
\begin{center}
	\begin{tikzpicture}
	\foreach \x in {0,1/3,2/3} {
		\draw[red,thick] ({\x*(-1/4)+(1-\x)*(-3/4)},{\x*(-1/4)+(1-\x)*(1)}) -- ({\x*(-1/4)+(1-\x)*(1)},{\x*(-1/4)+(1-\x)*(-3/4)}) -- ({\x*(-1/4)+(1-\x)*(1)},{\x*(-1/4)+(1-\x)*(1)}) -- cycle;
	}
	\draw[gray] (-1/4,-1/4) -- (1,-3/4);
	\draw[gray] (-1/4,-1/4) -- (-3/4,1);
	\draw[gray] (-1/4,-1/4) -- (1,1);
	\foreach \x in {-1,0,1} {
	\foreach \y in {-1,0,1} {
		\fill (\x,\y) circle (0.08);
	}
	}
	\end{tikzpicture} 
\end{center}
\caption{\label{fig:0<a<b} The relaxation complexity of $X_{1,2}$ is four. The figure presents the base and the two horizontal cross-sections of the tetrahedron $P$ that relaxes $X_{1,2}$.}
\end{figure} 
If $a=b$, then $\conv(X)$ is a tetrahedron, and so one has $\rc(X)=4$.
Thus, we assume $a < b$.
We consider the triangle $T$ with vertices 
\[
	(-1 + 2 t,1), (1,-1 + 2t), (1,1),
\]
where 
\[
	t = \frac{a}{3 b - 2a},
\]
and we check that the tetrahedron $P$ with base $T \times \{0\}$ and apex $p =  (-1/2,-1/2,3b/2)$ is a relaxation of $X_{a,b}$.
Note that $0  < t < 1$ so that the projection of $P$ onto~$\R^2$ is a quadrilateral that contains only four lattice points $0, e_1, e_2$ and $e_1+e_2$.
This shows that there are only three types of lattice points that can be contained in $P$, $(0,0,z)$, $(1,0,z)$, $(0,1,z)$ and $(1,1,z)$, with $z \in \Z_{\ge 0}$.
One can readily check that the points $(1,0,z)$, $(0,1,z)$ and $(1,1,z)$ belong to $P$ only for $z=0$.
Thus, it remains to check when $(0,0,z)$ belongs to $P$. Due to the symmetry of $(0,0,z)$ and $P$ with respect to the reflection $(x_1,x_2,x_3) \mapsto (x_2,x_1,x_3)$, it suffices to take the cross-section of $P$ by the hyperplane $x_1=x_2$ and check which points $(0,0,z)$ belong to this cross-section. The cross-section is a triangle with the vertices  $p, (1,1,0)$ and $(t,t,0)$. Thus, we need to check when the condition
\[
	(0,z) \in \conv\left( \{ (-1/2,3b/2), (1,0), (t,0) \} \right) =: C
\]
holds. The cross-section of the triangle $C$ by the line $x_2=0$ is the segment with endpoints 
\[
	(0,a) = \frac{2t}{2 t + 1} \cdot (-1/2, 3b/2) + \frac{1}{2 t + 1} \cdot (t,0)
\]
and
\[
	(0,b) = \frac{1}{3} \cdot (-1/2, 3b/2)  + \frac{2}{3} \cdot (1,0).
\]
This shows that $(0,0,z) \in P$ if and only if $z \in \{a,\ldots,b\}$, as desired (cf.~Figure~\ref{fig:0<a<b} for an illustration).

\smallskip
\noindent\emph{Case $a< 0 < b$:}
Consider the segment 
\[
	S = [-\epsilon,1],
\] where $\epsilon>0$ is sufficiently small.
We define the tetrahedron 
\[
	P = \conv \bigl(\underbrace{ S \times \{0\}}_{\text{horiz. segm.}} \times \underbrace{\{a\}}_{\text{height $a$}} \cup \underbrace{\{0\} \times S}_{\text{vert. segm.}} \times \underbrace{\{b\}}_{\text{height $b$}} \bigr)
\] and consider the horizontal cross-section of~$P$ at height $z$ by introducing 
\[
	P_z := \setcond{(x,y)}{(x,y,z) \in P}. 
\]
For $a \le z \le b$, $P_z$ is the rectangle given by 
\[
	P_z := \left( \frac{z-a}{b-a} \cdot S \right) \times \left( \frac{b-z}{b-a} \cdot S \right).
\]
In particular, the cross-section at height $0$ is given by 
\[
	P_0 = \left( -\frac{a}{b-a} \cdot S \right) \times \left( \frac{b}{b-a} \cdot S \right).
\]
We now define the two-dimensional lattice $\Lambda := \Z b_1 + \Z b_2 $ spanned by the vectors
\[
b_1 := \frac{1}{b-a} \cdot ( - a , - b \epsilon) \quad \text{and} \quad b_2 := \frac{1}{b-a} \cdot ( a \epsilon , b),
\]
which are two opposite vertices of the rectangle $P_0$. 
It turns out that, for a sufficiently small $\epsilon>0$, the set
\[
	X'_{a,b} := P \cap (\Lambda \times \Z)
\] 
is given as 
\[
	X'_{a,b} = \{b_1, b_2,b_1 + b_2\} \cup \setcond{z e_3}{z=a,\ldots,b}.
\]
Consequently, applying the transformation $\phi:\R^3 \to \R^3$, given by $\phi(b_1,0) = e_1$, $\phi(b_2,0) = e_2$, and $\phi(e_3) = e_3$, satisfying $\phi( \Lambda \times \Z) = \Z^3$, we conclude that
\[
	X_{a,b} = \phi(P) \cap \Z^3.
\]
As the relaxation $\phi(P)$ of $X_{a,b}$ is a tetrahedron as well, we are done.
\begin{figure}
\begin{center}
\begin{tikzpicture}[scale=2.5]
	\foreach \x in {1/3,2/3} {
		\filldraw[fill=black!20!white,opacity=0.5] ({\x*(-1/8)},{(1-\x)*(-1/8)}) -- ({\x*(-1/8)},1-\x) -- ((\x,1-\x) -- (\x, {(1-\x)*(-1/8)}) -- cycle;
	}
	\draw[very thick,red,->] (0,0) -- (1/3, -1/12);
	\draw[very thick,red,->] (0,0) -- (-1/24, 2/3); 
	\draw[thick] (-1/8,0) -- (1,0);
	\draw[thick] (0,-1/8) -- (0,1);
	\draw (-1/8,0) -- (0,-1/8) -- (1,0) -- (0,1) -- cycle;
	\foreach \x in {-2,-1,0,1,2,3}
\foreach \y in {-1,0,1,2}
{
	\fill[red] ({\x*(1/3)+\y*(-1/24)},{\x*(-1/12)+\y*(2/3)}) circle (0.03);
}
\end{tikzpicture} 
\end{center}
\caption{An illustration for the construction in the case $a < 0 < b$. The figure depicts the lattice $\Lambda$ (red dots), its generators $b_1,b_2$ (red vectors) and cross-sections of the tetrahedron $P$.}
\end{figure}
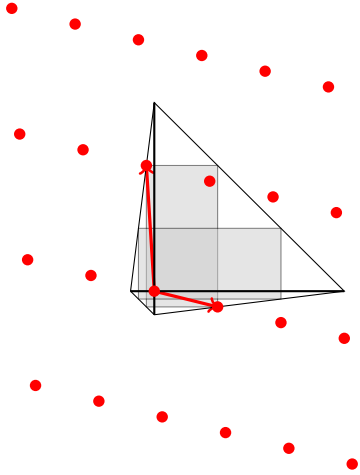 
\end{proof}

\subsection*{Details for Case 2.2.4}

If every lattice-width direction of~$X$ is of type $(2,2)$, then, in the notation from before, all the intersections $H_u^+ \cap X$ and $H_u^- \cap X$, $u \in \D_X$, are two-dimensional.
In view of Proposition~\ref{prop:finite-observers-dim-2}, there are only finitely many observers of $X$ in each of these hyperplanes and they can be computed algorithmically.
Using Lemma~\ref{lem:structure-of-observers}, we thus find an explicitly computable finite set $Y \subseteq \Z^3 \setminus X$ such that $\rc(X)=\rc(X,X \cup Y)$ and invoking Corollary~\ref{cor:X-Y-finite-computability} once more shows computablitity of~$\rc(X)$.

\bibliographystyle{amsplain}
\bibliography{mybib}

\providecommand{\bysame}{\leavevmode\hbox to3em{\hrulefill}\thinspace}
\providecommand{\MR}{\relax\ifhmode\unskip\space\fi MR }
\providecommand{\MRhref}[2]{%
  \href{http://www.ams.org/mathscinet-getitem?mr=#1}{#2}
}
\providecommand{\href}[2]{#2}
\begin{thebibliography}{10}

\bibitem{averkov2013lovasz}
Gennadiy Averkov, \emph{A proof of {L}ov\'{a}sz's theorem on maximal
  lattice-free sets}, Beitr. Algebra Geom. \textbf{54} (2013), no.~1, 105--109.

\bibitem{averkovconfortidelpiadisummafaenza2013onthe}
Gennadiy Averkov, Michele Conforti, Alberto Del~Pia, Marco Di~Summa, and Yuri
  Faenza, \emph{On the convergence of the affine hull of the
  {C}hv\'{a}tal-{G}omory closures}, SIAM J. Discrete Math. \textbf{27} (2013),
  no.~3, 1492--1502.

\bibitem{averkovkruempelmannnill2018lattice}
Gennadiy Averkov, Jan Kr\"umpelmann, and Benjamin Nill, \emph{{Lattice
  Simplices with a Fixed Positive Number of Interior Lattice Points: A Nearly
  Optimal Volume Bound}}, Int. Math. Res. Not. IMRN (2018).

\bibitem{averkovkruempelmannweltge2017notions}
Gennadiy Averkov, Jan Kr\"{u}mpelmann, and Stefan Weltge, \emph{Notions of
  maximality for integral lattice-free polyhedra: the case of dimension three},
  Math. Oper. Res. \textbf{42} (2017), no.~4, 1035--1062.

\bibitem{averkovwagner2012inequ}
Gennadiy Averkov and Christian Wagner, \emph{Inequalities for the lattice width
  of lattice-free convex sets in the plane}, Beitr. Algebra Geom. \textbf{53}
  (2012), no.~1, 1--23.

\bibitem{averkovwagnerweismantel2011maximal}
Gennadiy Averkov, Christian Wagner, and Robert Weismantel, \emph{{Maximal
  Lattice-Free Polyhedra: Finiteness and an Explicit Description in Dimension
  Three}}, Math. Oper. Res. \textbf{36} (2011), no.~4, 721--742.

\bibitem{basuconforticornuejolszambelli2010maximal}
Amitabh Basu, Michele Conforti, G\'{e}rard Cornu\'{e}jols, and Giacomo
  Zambelli, \emph{Maximal lattice-free convex sets in linear subspaces}, Math.
  Oper. Res. \textbf{35} (2010), no.~3, 704--720.

\bibitem{basupollackroy2006algorithms}
Saugata Basu, Richard Pollack, and Marie-Fran\c{c}oise Roy, \emph{{Algorithms
  in Real Algebraic Geometry}}, second ed., Algorithms and Computation in
  Mathematics, vol.~10, Springer-Verlag, Berlin, 2006.

\bibitem{blancohaasehofmannsantos2016finiteness}
M\'{o}nica Blanco, Christian Haase, Jan Hofmann, and Francisco Santos,
  \emph{{The Finiteness Threshold Width of Lattice Polytopes}},
  \url{https://arxiv.org/abs/1607.00798}, 2016.

\bibitem{blancosantos2016lattice}
M\'{o}nica Blanco and Francisco Santos, \emph{Lattice 3-polytopes with few
  lattice points}, SIAM J. Discrete Math. \textbf{30} (2016), no.~2, 669--686.

\bibitem{blichfeldt1914anew}
Hans~F. Blichfeldt, \emph{A new principle in the geometry of numbers, with some
  applications}, Trans. Amer. Math. Soc. \textbf{15} (1914), no.~3, 227--235.

\bibitem{boroshtreybig1976bounds}
Itshak Borosh and Leon~Bruce Treybig, \emph{Bounds on positive integral
  solutions of linear {D}iophantine equations}, Proc. Amer. Math. Soc.
  \textbf{55} (1976), no.~2, 299--304.

\bibitem{draismamcallisternill2012lattice}
Jan Draisma, Tyrrell~B. McAllister, and Benjamin Nill, \emph{Lattice-width
  directions and {M}inkowski's {$3^d$}-theorem}, SIAM J. Discrete Math.
  \textbf{26} (2012), no.~3, 1104--1107.

\bibitem{ginsburgspanier1964bounded}
Seymour Ginsburg and Edwin~H. Spanier, \emph{{Bounded ALGOL-Like Languages}},
  Trans. Amer. Math. Soc. \textbf{113} (1964), no.~2, 333--368.

\bibitem{gruber2007convex}
Peter~M. Gruber, \emph{Convex and {D}iscrete {G}eometry}, Grundlehren der
  Mathematischen Wissenschaften, vol. 336, Springer-Verlag, Berlin, 2007.

\bibitem{haase2018survival}
Christoph Haase, \emph{{A Survival Guide to Presburger Arithmetic}}, ACM SIGLOG
  News \textbf{5} (2018), no.~3, 67–82.

\bibitem{hammeribarakipeled1981threshold}
Peter~L. Hammer, Toshihide Ibaraki, and Uri~N. Peled, \emph{Threshold numbers
  and threshold completions}, Studies on graphs and discrete programming
  ({B}russels, 1979), Ann. Discrete Math., vol.~11, North-Holland,
  Amsterdam-New York, 1981, pp.~125--145.

\bibitem{hojny2018strong}
Christopher Hojny, \emph{{Strong IP Formulations Need Large Coefficients}},
  \url{http://www.optimization-online.org/DB_HTML/2018/11/6934.html}, 2018.

\bibitem{jeroslow1973therecannot}
Robert~G. Jeroslow, \emph{There cannot be any algorithm for integer programming
  with quadratic constraints}, Oper. Res. \textbf{21} (1973), 221--224.

\bibitem{jeroslow1975ondefining}
\bysame, \emph{On defining sets of vertices of the hypercube by linear
  inequalities}, Discrete Math. \textbf{11} (1975), 119--124.

\bibitem{kaibelweltge2015lowerbounds}
Volker Kaibel and Stefan Weltge, \emph{Lower bounds on the sizes of integer
  programs without additional variables}, Math. Program. \textbf{154} (2015),
  no.~1-2, Ser. B, 407--425.

\bibitem{liberti2019undecidability}
Leo Liberti, \emph{Undecidability and hardness in mixed-integer nonlinear
  programming}, RAIRO Oper. Res. \textbf{53} (2019), no.~1, 81--109.

\bibitem{lovasz1989geometry}
L\'{a}szl\'{o} Lov\'{a}sz, \emph{Geometry of numbers and integer programming},
  Mathematical programming ({T}okyo, 1988), Math. Appl. (Japanese Ser.),
  vol.~6, SCIPRESS, Tokyo, 1989, pp.~177--201.

\bibitem{nillziegler2011projecting}
Benjamin Nill and G\"unter~M. Ziegler, \emph{{Projecting Lattice Polytopes
  Without Interior Lattice Points}}, Math. Oper. Res. \textbf{36} (2011),
  no.~3, 462--467.

\bibitem{pikhurko2001lattice}
Oleg Pikhurko, \emph{Lattice points in lattice polytopes}, Mathematika
  \textbf{48} (2001), no.~1--2, 15--24.

\bibitem{presburger1929vollstaendigkeit}
Moj\.{z}esz Presburger, \emph{{{\"U}ber die Vollst{\"a}ndigkeit eines gewissen
  Systems der Arithmetik ganzer Zahlen, in welchem die Addition als einzige
  Operation hervortritt}}, Comptes Rendus du I congres de Mathematiciens des
  Pays Slaves (1929), 92--101.

\bibitem{rogersshephard1957the}
C.~Ambrose Rogers and Geoffrey~C. Shephard, \emph{The difference body of a
  convex body}, Arch. Math. (Basel) \textbf{8} (1957), 220--233.

\bibitem{schrijver1986theory}
Alexander Schrijver, \emph{Theory of linear and integer programming},
  Wiley-Interscience Series in Discrete Mathematics, John Wiley \& Sons, Ltd.,
  Chichester, 1986, A Wiley-Interscience Publication.

\bibitem{sperner1928einsatz}
Emanuel Sperner, \emph{Ein {S}atz \"{u}ber {U}ntermengen einer endlichen
  {M}enge}, Math. Z. \textbf{27} (1928), no.~1, 544--548.

\bibitem{tarski1951adecision}
Alfred Tarski, \emph{A decision method for elementary algebra and geometry},
  University of California Press, Berkeley and Los Angeles, Calif., 1951, 2nd
  ed.

\bibitem{taylorzwicker1999simplegames}
Alan~D. Taylor and William~S. Zwicker, \emph{Simple games: {D}esirability
  relations, trading, pseudoweightings}, Princeton University Press, Princeton,
  NJ, 1999.

\bibitem{weltge2015diss}
Stefan Weltge, \emph{Sizes of {L}inear {D}escriptions in {C}ombinatorial
  {O}ptimization}, Ph.D. thesis, Otto-von-Guericke-Universit{\"a}t Magdeburg,
  2015.

\end{thebibliography}

\end{document}